\documentclass[12pt]{amsart}

\usepackage{amsmath, amssymb, amsthm}

\numberwithin{equation}{section}

\usepackage{hyperref}
	\hypersetup{colorlinks, breaklinks,
            	linkcolor = blue,
				urlcolor = blue,
				anchorcolor = blue,
				citecolor = blue}

\usepackage{cleveref}

\usepackage{mathtools}

\DeclareMathOperator{\dist}{dist}
\DeclareMathOperator{\Isom}{Isom}
\DeclareMathOperator{\Sim}{Sim}
\DeclareMathOperator{\id}{id}
\DeclareMathOperator{\cl}{cl}

\renewcommand{\P}{\mathbb{P}}
\newcommand{\R}{\mathbb{R}}
\newcommand{\Q}{\mathbb{Q}}
\newcommand{\Z}{\mathbb{Z}}

\newcommand{\ii}{{\bf i}}
\newcommand{\jj}{{\bf j}}
\newcommand{\kk}{{\bf k}}
\newcommand{\mm}{{\bf m}}
\newcommand{\cH}{\mathcal{H}}

\newcommand{\emp}{\emptyset}

\renewcommand\epsilon{{\varepsilon}}
\newcommand\eps{{\varepsilon}}

\usepackage{thmtools}
\declaretheorem[numberwithin=section,name=Theorem]{theorem}
\declaretheorem[sibling=theorem,name=Lemma]{lemma}
\declaretheorem[sibling=theorem,name=Proposition]{proposition}
\declaretheorem[sibling=theorem,name=Corollary]{corollary}

\declaretheorem[sibling=theorem,name=Remark,style=remark]{remark}

\title[Kakeya, Besicovitch, Nikodym for rectifiable curves]{The Kakeya needle problem and the existence of Besicovitch and Nikodym sets for rectifiable sets}
\author{Alan Chang, Marianna Cs\"ornyei}
\address{Department of Mathematics, The University of Chicago, 5734 S. University Avenue, Chicago, IL 60637, USA}

\subjclass[2010]{	28A75} 
\keywords{Kakeya needle problem, Besicovitch set, Nikodym set, rectifiable set}

\email{ac@math.uchicago.edu}
\email{csornyei@math.uchicago.edu}

\begin{document}

\begin{abstract}
We solve the Kakeya needle problem and construct a Besicovitch and a Nikodym set for rectifiable sets.
\end{abstract}

\maketitle

\section{Introduction}
Let $E\subset\R^2$ be a rectifiable set. Our aim in this paper is to show that the classical results about rotating a line segment in arbitrarily small area, and the existence of a Besicovitch and a Nikodym set hold if we replace the line by the set $E$. We will explain our results in more details below, but first we present two illustrative examples.

\begin{enumerate}
\item
If $E$ is the graph of a convex function $f:\,\R\to\R$, our results imply the following:
\emph{E can be rotated continuously by $360^\circ$ covering only a set of zero Lebesgue measure, if at each time moment $t$ we are allowed to delete just one point from the rotated copy of $E$.} 
\item 
If $f$ is not just convex but strictly convex, then: \emph{E can be moved continuously, using only translations, to any other shifted position, covering a set of measure zero, if at each time moment $t$ we are allowed to delete just one point from the translated copy of $E$.}
\end{enumerate}

\begin{remark} In the two examples given above, our movement $t\mapsto E_t $ is continuous, but the point $x_t \in E_t$ we delete cannot be chosen continuously. However, all our constructions in this paper are Borel. 

In the first example, if we take $E$ to be a general rectifiable set, the result still holds, but instead of a single point, we need to delete an $\cH^1$-null subset of $E$ (see \Cref{theorem:besicovitch-rot}). For the generalization of the second example to rectifiable sets, see \Cref{theorem:besicovitch}.

\end{remark}

In the first example, $\bigcup_t (E_t\setminus \{x_t\})$ is Lebesgue null. Therefore, $\bigcup_t (E_t\setminus \{x_t\})$ is a \emph{Besicovitch set}: in each direction it contains not just a ``unit line segment of the line $E$'' but a whole copy of the set $E$ except for one of its points. 

On the other hand, since $\bigcup_t E_t$ has non-empty interior, we can cover $\R^2$ by taking a countable union of copies of $\bigcup_t E_t$. Therefore, the countable union of copies of $\bigcup_t(E_t\setminus \{x_t\})$ is a \emph{Nikodym set}: it has measure zero, and through each point $x\in \R^2$, it contains a copy of the set $E$ with one point removed.

For the case when $E$ is a line, see, e.g., \cite{mattila15} for both classical and recent results.

\subsection{History}\label{subsection:history}
The Kakeya needle problem for sets other than the line segment has been studied before. 
R.O. Davies proved in \cite{davies} that not only one but any finite union of parallel line segments can be rotated by $360^\circ$ covering arbitrarily small area. He also showed that the line segments must be parallel: if a set contains two line segments that are not parallel to each other, then it can no longer be moved.

\medskip
In \cite{chl} the authors introduced the following definitions: a planar set $E$ has the \emph{Kakeya property} if there exist two different positions of $E$ such that $E$ can be moved continuously from the first position to the second in such a way that the area covered by $E$ along the movement is arbitrarily small. A planar set $E$ has the \emph{strong Kakeya property} if it can be moved in the plane continuously to any other shifted or rotated position in a set of arbitrarily small area.

In \cite{chl} it is shown that if $E$ is a closed connected set that has the Kakeya property, then $E$ must be a subset of a line or of a circle. Moreover, if $E$ is an arbitrary closed set that has the Kakeya property, then the union of the non-trivial connected components of $E$ must be a subset of parallel lines or of concentric circles. 

In \cite{hl} the authors show that short enough circular arcs of the unit circle possess the strong Kakeya property. (For topological reasons, it is clear that a full circle does not have the strong Kakeya property.)

\subsection{Translations}
Let us consider a related question for circular arcs: can we \emph{translate} a full circle continuously to any other position covering arbitrarily small area, if at each point of the translation, we are allowed to delete an arc of the circle of a given length? How long must the deleted arc be? Because of rotational symmetry, the question of which circular arcs have the strong Kakeya property is equivalent to this one, as long as we choose the deleted arc piecewise continuously.

In this paper, we will answer this ``piecewise continuous question'' for an arbitrary rectifiable set $E$ of finite $\cH^1$-measure in the following way: we only need to delete points whose tangent directions lie in a small interval.

\medskip
Let us state our results precisely. We will use the following notation and terminology.

We let $\P^1 \simeq \R/\pi\Z$ denote the set of all directions in $\R^2$. We will use the standard embedding of $\R^2$ into the projective plane $\P^2$, so that $\P^2 = \R^2 \cup \P^1$. The arc-length metric on the unit sphere $S^2$  together with the quotient map $S^2 \to \P^2$ gives us a metric on $\P^2$. Let $(\P^2)^*$ denote all the lines in $\P^2$.

 We denote by $|\cdot|$ the Lebesgue measure on $\R^2$ or $\P^1$, and by $\cH^1$  the $1$-dimensional Hausdorff measure on $\R^2$. As usual, $B(x,r)$ denotes the open ball centered at $x$ of radius $r$, and $B(S, r)$ denotes the open $r$-neighborhood of a set $S$. We denote by $\cl S$ the closure of $S$. We write $A \lesssim B$ to mean $A \leq CB$ for some absolute constant $C > 0$.
 
Recall that every rectifiable set $E \subset \R^2$ has a tangent field, which is defined for $\cH^1$-almost every $x \in E$ (see \Cref{sec:rectifiable}).  We let $\theta_x \in \P^1$ denote the tangent of $E$ at $x$, and we let $\nu_x \in (\P^2)^*$ denote the normal line of $E$ at $x$. (The direction of $\nu_x$ is the one orthogonal to $\theta_x$.) Note that $\nu_x$ is the normal \emph{line} passing through the point $x$, and not just a normal vector.

We will start by proving the following theorem:

\begin{theorem}[Kakeya needle problem for translations]
\label{theorem:translations}
Let $E \subset \R^2$ be a rectifiable set of finite $\cH^1$-measure. Let $\epsilon > 0$ be arbitrary. Then between the origin and any prescribed point in $\R^2$, there exists a polygonal path $P = \bigcup_{i=1}^n L_i$ with each $L_i$ a line segment, and for each $i$ there exists a direction $\theta_i\in\P^1$, such that
\begin{equation}
\label{eq:theorem-translations}
|\bigcup_{i}\bigcup_{p\in L_i} (p+\{x\in E:\,\theta_x\not\in B(\theta_i,\eps)\})|< \epsilon.
\end{equation}
\end{theorem}

Although the tangent field of a rectifiable set is defined only $\cH^1$-almost everywhere, for the statement of \Cref{theorem:translations} (and for all other results in this paper), we need to define it pointwise. We will show that regardless of which pointwise representation we choose, the results remain true (see \Cref{sec:rectifiable}).

\Cref{theorem:translations} has an immediate corollary:

\begin{corollary}
\label{corollary:circle}
If we remove an arbitrary neighborhood of two diametrically opposite points from a circle, the resulting set can be moved continuously to any other position in the plane in arbitrarily small area.
\end{corollary}

This strengthens the previously known result \cite{hl} that sufficiently short circular arcs have the strong Kakeya property. 

\subsection{Rotations}
We note that \Cref{theorem:translations} does not handle the classical Kakeya needle problem: clearly it is not possible to translate a line segment to every other position in small area. We can still apply \Cref{theorem:translations} with $E$ a line segment, but since every point of $E$ has the same tangent direction, it allows us to delete the entire line segment at every point $p \in P$. To obtain a more meaningful statement for line segments, we need to 
consider what happens if we allow rotations as well as translations. 

\medskip

In order to unify translations and rotations, it is helpful to consider the projective plane $\P^2$.

We can consider a translation in direction $\theta \in \P^1$ to be a ``rotation'' around the infinite point $\theta^\perp \in \P^1 \subset \P^2$ (see \Cref{rott}).

We need to generalize the notion of a polygonal path from a path in $\R^2$ to one in $\Isom^+(\R^2)$, the space of all orientation-preserving isometries of $\R^2$. (This space is also known as the \emph{special Euclidean group} $SE(2)$.) The polygonal path in \Cref{theorem:translations} can be viewed as a sequence of vectors, each indicating in which direction and how far to translate. Then, a polygonal path of rotations should be a sequence of rotations, indicating around which point and how much to rotate.

Specifying a sequence of rotations is slightly trickier than a sequence of translations: when we rotate a set around a point, the centers of all the other rotations move. To avoid this problem, we will find it much more convenient to specify our sequence in the \emph{intrinsic coordinate system}. That is, with $\rho_i$ denoting rotations around $z_i\in\R^2$, our continuous movement will be to rotate first with center $z_1$, then with center $\rho_1(z_2)$, and so on. 

Our \emph{polygonal path} $P$ will be specified by the \emph{intrinsic sequence} $\rho_i$, but it will still lie in the space $\Isom^+(\R^2)$, and its points will be isometries not in the intrinsic but in the standard coordinate system. 

For each sequence $\{\rho_i\}$ we obtain a $P=\bigcup_i L_i$. For each ``line segment'' $L_i$ in $P$, the rotations in $\{p'\circ p^{-1} : p, p' \in L_i\}$ all have the same center. (It is important to remember that this center depends not only on $z_i$ but also on the previous rotations.) 

Also, we find it much more convenient to specify a rotation not by the point that we rotate around, but by the image of this point in the projective plane when we embed $\R^2$ into $\P^2$. We will call this the \emph{projective center} of $\rho$ (both for translations and rotations).

\medskip
First we will prove a preliminary result (see \Cref{theorem:rotations-on-line}). The exact statement is quite technical, but essentially says that instead of using translations, we can move our set $E$ using rotations whose projective centers are almost aligned: if we want to connect $\rho\in\Isom^+(\R^2)$ to the identity map by a polygonal path, we can choose a line $\ell\in(\P^2)^*$ that passes through the center of $\rho$, and choose the (intrinsic) rotations so that their projective centers lie in $B(\ell, \eps)$. We also obtain, for each $i$, a $u_i \in \ell$ such that:
\begin{equation}
\label{eq:throw-out-belt}
|\bigcup_i\bigcup_{p\in L_i}p(\{x\in E:\, \nu_x \cap \ell \cap B(u_i, \eps)=\emptyset\})|<\eps.
\end{equation}

\Cref{theorem:translations} can be viewed as a special case of \Cref{theorem:rotations-on-line} by taking $\rho$ to be a translation and then taking $\ell$ to be $\P^1$. In this case, the centers lie on $\ell$, not just in an $\epsilon$-neighborhood of $\ell$. The reason we need an $\epsilon$-neighborhood for rotations is that, unlike for translations,
the composition of a rotation around  $z_1$ and a rotation around $z_2$ does not equal a rotation around  a point on the line through $z_1,z_2$. (Recall the centers are specified with intrinsic coordinates.) This makes the statement and the proof of \Cref{theorem:rotations-on-line} more complicated than those of \Cref{theorem:translations}. We will need careful error estimates on how far the centers move, and, consequently, how large area the set $E$ covers during its movement.

The essential observation for the error estimates is the following: the composition structure of translations is linear, i.e., given by vector addition. The composition structure of rotations is not linear, but it is ``linear up to a quadratic error,'' using an appropriate parametrization of $\Isom^+(\R^2)$ (see \Cref{lemma:star-approx-add}).

\begin{remark}
\label{remark:classical-kakeya}
Let $E$ be a countable union of parallel line segments which is bounded and has finite total length. It is easy to see that \Cref{theorem:rotations-on-line} implies that we can rotate $E$ inside a set of arbitrarily small area. This strengthens the result of Davies mentioned at the beginning of this introduction, who proved the same result when $E$ is a finite union of parallel line segments \cite{davies}.
\end{remark}

\begin{remark} 
\Cref{theorem:translations} and \Cref{theorem:rotations-on-line} also provide a new insight into the other results mentioned in \Cref{subsection:history}, that the non-trivial connected components of a closed set with the Kakeya property can be covered by parallel lines or by concentric circles \cite{chl}. It turns out that the key property of the line and the circle is that they are homogeneous: by rotating around the center of the circle, any sub-arc can be mapped onto any other sub-arc of the same length, by a continuous movement that covers only zero area. The same is true for lines with shifts. Therefore our piecewise continuous deletion of the line segments/sub-arcs in
\Cref{theorem:translations}/\Cref{theorem:rotations-on-line} can be replaced by a continuous one. No rectifiable set except for the union of parallel lines or concentric circles has this property.
\end{remark}

\begin{remark}\label{unbounded}
The set $E$ in \Cref{theorem:rotations-on-line} needs to be bounded. Take, for example, $E$ to be a union of countably many circles with centers $z_i$ and radius $r_i$, such that $\sum {r_i} < \infty$ and $\sum r_i |z_i| = \infty$. Then it is a rectifiable set with finite $\cH^1$-measure, but every continuous rotation with a fixed center covers infinite area, even with a normal line removed. However, for the limit version \Cref{theorem:besicovitch-rot} (explained below), in which the centers of the rotations no longer need to be piecewise constant along the path, we can drop the boundedness condition.
\end{remark}

\subsection{Besicovitch and Nikodym sets}
In \Cref{section:besi-niko}, we study what happens in the limit as $\eps \to 0$. By taking a sequence of $\eps$ tending to zero, the balls $B(u_i, \eps)$ shrink to a single point in $\P^2$, and the area covered shrinks to zero. We obtain in the limit a continuous movement $P \subset \Isom^+(\R^2)$ such that the set $E$ covers only \emph{zero} area, where at each time moment we only need to delete a subset of $\cH^1$-measure zero (see \Cref{theorem:besicovitch-rot}). The resulting set of zero area is an analogue of a Besicovitch set for $E$. 

Consider, e.g., the special case where there is a line $\ell \in (\P^2)^*$ such that there is a neighborhood of $\ell$ in which no two normal lines of $E$ intersect. Then \Cref{theorem:besicovitch-rot} says that we can rotate $E$ continuously by $360^\circ$, covering a set of zero Lebesgue measure, where at each time moment, we only need to delete \emph{one point}. This happens, e.g., in the special case when $E$ is the graph of a convex function; by  choosing the line $\ell$ to lie below the graph, there is a neighborhood of $\ell$ where no two normal lines meet. 

If $E$ is strictly convex, then we can apply \Cref{theorem:besicovitch-rot} with $\ell=\P^1$ and hence translate $E$ to an arbitrary position in the plane in a set of zero Lebesgue measure, deleting one point at each time moment. 

For moving a circle, we can choose any line $\ell$. In this case we need to delete, at each time moment, not just one but two diametrically opposite points of the circle, since they have the same normal line.

\medskip

By the continuity of $P$, we can construct, from these Besicovitch sets, analogues of Nikodym sets, using the technique outlined at the beginning of this introduction. We state these more precisely in \Cref{section:besi-niko}.

\begin{remark} 
There is not only one continuous $P$, but residually many, in the sense of Baire category (see \Cref{remark:residual}). For results of similar nature when $E$ is a line, see, e.g., \cite{korner} and \cite{cchk}. 
\end{remark}

\begin{remark} It is well-known that there are no sets in $\R^n$ ($n\geq 2$) which have measure zero and contain a circle centered at every point. Stein first proved this for $n \geq 3$ by his estimates on spherical maximal functions \cite{stein}. Bourgain and Marstrand independently showed the same non-existence result holds for $n = 2$ around the same time \cite{bourgain, marstrand}. Bourgain's paper actually treats smooth curves with non-vanishing curvature. More work has been done on such curves, e.g., \cite{mitsis, wolff97, wolff00}.

The non-existence results concern placing a copy of $E$ around every point in $\R^2$. For our Nikodym result, we instead place a copy of $E$ \emph{through} every point of $\R^2$. With this change, such a construction is now possible.
\end{remark}

\begin{remark}
Somewhat surprisingly, ``Besicovitch sets'' for rectifiable sets in $\mathbb R^2$ do not necessarily have dimension $2$. 

Trivially, if $E$ is a countable union of concentric circles, then we can rotate them around their common center without increasing the dimension. More interestingly, there are also other, less trivial examples. For example, if $E$ is a countable union of circles (not necessarily concentric ones), then there is a $1$-dimensional set which contains a rotated copy of $E$ in each direction: since a residual set contains a shifted copy of any countable configuration of points, by putting countably many circles around the points of a $0$-dimensional residual set (of the same sizes as the circles in $E$), we obtain a $1$-dimensional Besicovitch set for $E$.
\end{remark}

\subsection{The sharpness of our results, and dilations}
We do not know whether the sizes of the sets we delete are sharp. While \Cref{theorem:besicovitch-rot} tells us that we only need to delete an $\cH^1$-nullset of $E$ at each time moment, perhaps it is possible to delete much fewer points than specified by the theorem. (For more precise information on the size of the sets we delete, see also \Cref{propo}, \Cref{rempre1}, \Cref{propo-rot} and \Cref{rempre2}.) 

For example, it would be interesting to know whether it is possible to translate a circle in a set of Lebesgue measure zero, deleting only one point at every time moment. This is still an open problem. 

\medskip
Cunningham proved that if we remove an arbitrary neighborhood of one point from a circle, the resulting circular arc can be \emph{shrunken} to a point using translations, rotations, and dilations \cite{cunningham3}. His construction is based on the classical straight line results, and the stereographic projection between the plane and the sphere; it works only for circles.

Motivated by Cunningham's result, in \Cref{sec:dilations}, we include a brief note about what happens for general rectifiable sets when we consider the space $\Sim^+(\R^2)$ of all orientation-preserving similarity transformations of $\R^2$. This space consists of $\Isom^+(\R^2)$ as well as dilations and transformations which rotate and dilate simultaneously, moving points along logarithmic spirals. The techniques in \Cref{sec:rotations} carry over to this setting because the composition structure of $\Sim^+(\R^2)$, like that of $\Isom^+(\R^2)$, is ``linear up to a quadratic error.''

In \eqref{eq:throw-out-belt}, we consider the intersection of the normal lines $\nu_x$ with balls in the projective plane. In the similarity transformations setting, we need to consider instead the intersections with \emph{rotated} normal lines. The angle by which we rotate normal lines depends on the ``pitch angle'' of the logarithmic spirals of the similarity transformations.

As an illustrative example, in \Cref{subsection:circles} we present the following application. (Compare this with \Cref{corollary:circle}.)

\begin{corollary}
\label{corollary:circles-dilation}
Any circular arc which is not the full circle can be moved continuously to any other position in the plane (of the same size) in arbitrarily small area via similarity transformations, such that the size of the circular arc always remains arbitrarily close its initial size.
\end{corollary}

We also obtain a Nikodym set for circles, i.e., a set in the plane of Lebesgue measure zero which contains a punctured circle through every point. 

\begin{corollary}
\label{corollary:nikodym-circles-dilations}
There exists a set $A \subset \R^2$ of Lebesgue measure zero such that for each $x \in \R^2$, there is a circle $C$ such that $x\in C$ and $C\setminus A$ has at most one point.
\end{corollary}

\section{Main ideas of the proof of Theorems \ref{theorem:translations} and
\ref{theorem:rotations-on-line}}

Our proof of \Cref{theorem:translations} 
relies on two key ideas. 

\subsection{The first key idea}

Our first key idea is the ``small neighborhood lemma'': suppose we move a compact set $E \subset \R^2$ along a path of isometries $P \subset \Isom^+(\R^2)$. If we perturb $P$ by a small amount, the area covered by the perturbed movement will not increase very much because the new region covered is contained in a small neighborhood of the original. This simple and obvious fact turns out to be extremely useful. %

\begin{lemma}[Small neighborhood lemma]
\label{lemma:small-nbhd-small}
Let $E \subset \R^2$ be any compact set, and let $P$ be an arbitrary path in $\Isom^+(\R^2)$. Then for every $\epsilon > 0$ there exists a neighborhood $U \subset \Isom^+(\R^2)$ of $P$ such that
\begin{equation}
|\bigcup_{p \in U} p(E)| \leq |\bigcup_{p \in P} p(E)| + \epsilon.
\end{equation}
\end{lemma}

(In this paper, a path is the image of a continuous map on a compact interval.)

\subsection{The second key idea}

The second key idea is more technical, so we give only an informal presentation here and defer the precise details to \Cref{sec:second-key-idea-lines} and \Cref{sec:second-key-idea-rot}. First, note that: 

\begin{lemma}\label{lemma:luzin}
For any polygonal path $P\subset \R^2$ and for an arbitrary $E\subset\R^2$, if we translate $E$ along $P$, then the area covered is 
$\lesssim\cH^1(E)\cH^1(P)$. 
\end{lemma}

\begin{proof}
If $B$ is a ball of radius $r$, where $r$ is smaller than the line segments in the polygonal path, then for each line segment $L\subset P$, by translating $B$ along $L$ we cover a set of area $\lesssim r\cH^1(L)$. Adding these up for all line segments $L$ and approximating $E$ by a union of small balls, we obtain \Cref{lemma:luzin}.
\end{proof}

\begin{remark}\label{remark:luzinrem} \Cref{lemma:luzin} shows that in our proof of \Cref{theorem:translations} we can ignore small subsets of $E$, since in the movements these will cover only small area. Also, we can ignore small subsets of $P$.
\end{remark}
 
Our second key idea is the simple observation that the estimate in \Cref{lemma:luzin} can be improved if we also take into account the directions of the tangents of $E$. For simplicity, suppose that $E$ is a $C^1$ curve. Then we can cover $E$ with thin rectangles that approximate the curve. Each thin rectangle $R$ has the property that translating $E \cap R$ along a line segment $L$ in the direction of the long side of $R$ covers area $\lesssim \delta\cH^1(L)$, where $\delta$ is the length of the short side of the rectangle. If the rectangle is thin enough, then this is a much better estimate than the estimate $\cH^1(E\cap R)\cH^1(L)$ that we obtain from \Cref{lemma:luzin}.

\begin{remark} 
For general rectifiable sets $E$, instead of thin rectangles, we will choose $R\subset E$ such that $\theta_x$ is almost constant on $R$. The key idea remains the same (see \Cref{sec:second-key-idea-lines}).
\end{remark}

\subsection{Combining the key ideas}

We combine these two ``key ideas'' to construct polygonal paths in a Venetian blind-type construction. (For Venetian blinds, see, e.g., \cite[Theorem 6.9]{falconer2014} or \cite[Lemma 11.8]{mattila15}.) Again, we give an informal presentation. See \Cref{sec:translations} for the precise details.

The method is as follows. Suppose that $E$ is a $C^1$ curve, which we cover by thin rectangles. Suppose our initial path is a translation along a horizontal segment. Let $R, R'$ be two rectangles from our cover and let $\theta, \theta'$ be the directions of their long sides, with $\theta \neq \theta'$.

\begin{enumerate}
\item
First, we replace our horizontal segment by a zigzag so that every other segment has direction $\theta$. Then $R \cap E$ will cover small area when translated along these segments.
\item
  Now we repeat the previous step, replacing each segment in direction $\theta$ with a new zigzag such that every other segment has direction $\theta'$. Then $R' \cap E$ will cover small area when translated along the segments in direction $\theta'$. Furthermore, if we make these new zigzags sufficiently ``fine'' (many turns and small enough segments), then these zigzags will remain close to the segments of direction $\theta$ that we just replaced. Then by the small neighborhood lemma, $R \cap E$ also covers small area when moved along the segments in direction $\theta'$.
\end{enumerate}

By the end of step (2), we now have a ``Venetian blind.'' The line segments in direction $\theta'$ are the ``good'' segments, because translations along these segments cover small area for both rectangles $R$ and $R'$. By iterating with more angles, we can increase the number of rectangles for which translations along the good segments cover small area. 

We also need the total length of the remaining ``bad'' segments to be strictly smaller than the initial segment, so that the size of the bad segments tends to zero when we iterate the Venetian blind construction. (For this reason, we cannot deviate too far from the initial horizontal direction. This leads to condition \eqref{rat}.) After sufficiently many iterations, we can ignore the bad segments by \Cref{remark:luzinrem}.

\medskip
The main ideas of the proof of  \Cref{theorem:rotations-on-line}, where we use rotations, are similar. As in the proof for translations, we combine the small neighborhood lemma with the covering of $E$ by sets $R$ such that  rotating $R$ around an appropriate point $z$ covers only a small area.
We still use a Venetian blind construction, but now our zigzags will be in $\Isom^+(\R^2)$. The general ideas of the argument are the same, but, as we explained in the introduction, they will require more delicate estimates than for translations.

\section{Preliminaries}

\subsection{Tangents of rectifiable sets}
\label{sec:rectifiable}

Recall that a set $E\subset\R^2$ is called \emph{rectifiable} if $\cH^1$-a.e.\ point of $E$ can be covered by countably many $C^1$ curves. For any two $C^1$ curves, their tangent directions agree at $\cH^1$-a.e. point of their intersection. Therefore, there exists a tangent field to $E$, i.e., a map $x \mapsto \theta_x$ from $E$ to $\P^1$ such that for any $C^1$ curve $\Gamma$, the tangent direction to $\Gamma$ at $x$ agrees with $\theta_x$ for $\cH^1$-a.e.\ $x \in \Gamma \cap E$. This gives one of the (many equivalent) descriptions of a tangent field of a rectifiable set. 

Of course, the tangent field is uniquely defined only up to an $\cH^1$-null subset of $E$. That is, if we change the tangent field along a set that meets each $C^1$ curve in a set of $\cH^1$-measure zero, it is still a tangent field.

In order to prove \Cref{theorem:translations} and also our other results, we fix a particular tangent field $x \mapsto \tilde\theta_x$ on $E$ as follows:
first we fix a subset $E'\subset E$ of full $\cH^1$-measure and a cover of $E'$
by countably many $C^1$ curves $\{ \Gamma_i \}$. Next, for each $x \in E$, if all the curves $\Gamma_i$ that go through $x$ have the same tangent direction at that point, then we let $\tilde\theta_x \in \P^1$ be that direction. (This also defines the normal line $\tilde\nu_x$ at $x$.)

Consider the set of those $x \in E$ where our $\tilde\theta_x,\tilde\nu_x$ either (1) are not defined, or (2) are defined but do not agree with the $\theta_x$, $\nu_x$ from the statements of our theorems. This is a set of zero $\cH^1$-measure; hence we can ignore it by \Cref{remark:luzinrem} when we work with translations, and we will be able to ignore it by \Cref{lemma:luzin-rot} (below) when we work with rotations. Hence, for the remainder of this paper, we may assume that $\theta_x$ is the particular tangent field $\tilde\theta_x$ from the previous paragraph (and make the analogous assumption for $\nu_x$).

\subsection{Rotations}\label{rott}

We denote by $\Isom^+(\R^2)$ the space of all orientation preserving isometries of $\R^2$. Each element of $\Isom^+(\R^2)$ is either a translation by a vector $v$, or a rotation around a point $z\in\R^2$ by angle $\phi$. Using complex notation, such a rotation is the map $u\mapsto e^{i\phi}(u-z)+z$. The image of $0$ under this mapping is $z(1-e^{i\phi})$, so it is natural to denote 
\begin{equation}\label{v}
v:=z(1-e^{i\phi}).
\end{equation} 
We can see from \eqref{v} that $v=z(-i\phi+O(\phi^2))$. We denote 
\begin{equation}
w=\begin{cases}
z\phi & \phi\neq 0\\
      iv& \phi=0.\\
\end{cases}
\end{equation}
The motivation behind our notation is that, for small $\phi$ and near the origin, the rotation acts, to first order, like translation by $-iw$. 

Both translations and rotations can now be specified by an ordered pair $(w, \phi) \in \R^2 \times \R$. (We have $\phi=0$ for translations.) From now on, we will refer to translations as rotations as well.

For $\rho\neq\id$, we define the \emph{projective center} of $\rho$ to be the image of $(w, \phi)$ under the quotient map $\R^3 \setminus \{0\}\to\P^2$. We still use $(w, \phi)$ to denote the image in $\P^2$. If $\phi \neq 0$, then, using homogeneous coordinates, this reduces to $(z\phi,\phi)=(z, 1)$, as expected. If $\phi = 0$, then $(w, 0) = (iv, 0)=(v^\perp,0)$, which is indeed the point at infinity orthogonal to the direction of $v$. 

\begin{remark}
Even though we now view translations as rotations around infinite points, translations and rotations are still different, even when viewed in $\P^2$. For example, a rotation with angle $\phi\neq 0$ fixes just one point in $\P^2$ (its projective center) whereas a translation fixes an entire line ($\P^1\subset\P^2$).
\end{remark}

We will use the notation $\rho(w,\phi)$ for a rotation whose projective center is $(w,\phi)\in\P^2$ and whose angle is $\phi$.
That is, we assign to each point $x=(x_1,x_2,x_3)\in \R^3\setminus\{0\}$ the rotation $\rho(x)$ whose:
\begin{itemize}
\item projective center is the image of $(x_1,x_2,x_3)$ under the projection $\R^3\setminus\{0\} \to\P^2$;
\item angle is the last coordinate $x_3$.
\end{itemize}

\begin{remark} 
We will use the same notation $\rho=\rho(w,\phi)$ for the mapping $\rho : \R^2\to \R^2$ and for the continuous movement that rotates $\R^2$ around a point. For example, if $\phi=2\pi$ then the former is the identity mapping and the latter is not. %
It will be always clear from the context which one we mean.
\end{remark}

\subsection{The ``second key idea'' for translations}
\label{sec:second-key-idea-lines}

We fix a small $\delta>0$, and a direction $\theta$. Let $R$ be a subset of $E$ such that $|\theta_x-\theta|\lesssim\delta$ for every $x\in R$. Our aim is to estimate how large area we cover if we translate $R$ by a vector $v$ of direction $\theta$.

For each $x\in R$ there is a $C^1$ curve $\Gamma_i$ from \Cref{sec:rectifiable} that goes through the point $x$.
We choose a decomposition $R=\bigcup R_i$ such that $R_i\subset\Gamma_i$ for each $i$.  Then locally, i.e., in a neighborhood of $x\in R_i$, $\Gamma_i$ is the graph of a Lipschitz function $f$ in the $(\theta,\theta^\perp)$ coordinate system, with Lipschitz constant $\lesssim\delta$. Without loss of generality we can assume that $\theta=0$. 

Now, when we translate $R_i$ by the horizontal vector $v$, for each fixed $t\in\R$ we obtain $\#\{x\in\R:\,f(x)=t,\,(x,f(x))\in R_i\}$ many (not necessary disjoint) horizontal line segments on the line $y=t$, each of length $|v|$.
Therefore, by Fubini's theorem, the area covered is at most
\begin{equation}\label{eq:fubini-bound}
|v|\int \#\{x\in\R:\,f(x)=t,\,(x,f(x))\in R_i\}\,dt.
\end{equation}

Next, recall that the coarea formula (e.g., \cite[Theorem 3.2.22]{federer}) implies that for any measurable function $g: \R \to \R$, for any Lipschitz $h: \R \to \R$, and for any measurable $S \subset \R$,  
\begin{equation}\label{eq:coarea-formula}
\int_{\R} g(t) \#\{x \in S: h(x) = t\} \, dt
=
\int_S g(h(x)) |h'(x)| \, dx
.
\end{equation}

Using \eqref{eq:coarea-formula}, we can bound \eqref{eq:fubini-bound} by
\[
\lesssim \delta\cH^1(R_i)|v|
\]
since $f$ has Lipschitz constant $\lesssim \delta$. Summing over $i$, we obtain the following:
\begin{lemma}\label{lemma:variation}
Let $\delta>0$ be sufficiently small, and let $\theta$ be an arbitrary direction. Let $R$ be a subset of $E$ such that $|\theta_x-\theta|\lesssim \delta$ for every $x\in R$. Then if we translate $R$ by a vector $v$ of direction $\theta$, the total area covered is $\lesssim \delta\cH^1(R)|v|$. 
\end{lemma}

Note that $|\theta_x-\theta|\lesssim \delta$ if and only if $\nu_x$ meets a $\lesssim \delta$-neighborhood of $\theta^\perp$ in $\P^1$. Using this observation, we generalize \Cref{lemma:variation} to rotations in the next section.

\subsection{The ``second key idea'' for rotations}
\label{sec:second-key-idea-rot}

Let $z\in\R^2$, and let $\phi$ be an arbitrary angle. If we rotate the set $E$ around the center $z$ by angle $\phi$, then each point $x\in E$ moves along a circular arc of length $|x-z||\phi|$. Therefore, the trivial estimate we get is that by rotating $E$, the area covered is
\begin{equation}\label{trivirot}
\leq|\phi|\int_\R r\,\#\{x\in E:\,|x-z|=r\}\,dr\le |\phi|\int_E |x-z|\,d\,\cH^1(x).
\end{equation}
The first inequality follows from Fubini's theorem. The second follows from the coarea formula \eqref{eq:coarea-formula} and the fact that if we parametrize the curve $\Gamma_i$ by arc-length, the mapping $t\mapsto |x(t)-z|$ is Lipschitz, with Lipschitz constant at most 1.

For a general rectifiable set, the right-hand side of \eqref{trivirot} can be infinite (cf.\ \Cref{unbounded}). From now on, in this section we assume that $E$ is bounded. More precisely, we assume that $E\subset B(0,r)\subset\R^2$ (here, we used the Euclidean metric). We will show that there is a constant $c$ (that depends only on $r$) such that the following two lemmas hold.

\begin{lemma}\label{lemma:luzin-rot}
Let $y = (w, \phi) \in \R^2 \times \R$, let $\rho=\rho(y)$ be a rotation, and let $R\subset E$ be arbitrary. Then, if we rotate $R$ by $\rho$, the area covered is 
$\lesssim c \cH^1(R)|y|$.
\end{lemma}

\begin{lemma}\label{var2}
Let $\delta>0$ be sufficiently small (depending on $r$).
Let $y = (w, \phi) \in \R^2 \times \R$ and let $\rho=\rho(y)$ be a rotation with projective center $z$. Let $R\subset E$ be such that, for each $x\in R$, $\nu_x\cap B(z,\delta)\neq \emptyset$. (Here, the ball $B(z, \delta)$ is defined with respect to the metric on $\P^2$.) Then, when we rotate $R$ by $\rho$, the area covered is 
$$\lesssim c\delta\cH^1(R)|y|.$$
\end{lemma}

\begin{proof}[Proof of \Cref{lemma:luzin-rot}]
By \Cref{lemma:luzin}, we know that \Cref{lemma:luzin-rot} holds (with $c=1$) when $\rho$ is a translation. Now suppose that $z\in\R^2$ and $\phi\neq 0$. Then there is a constant $c_1$ (that depends only on $r$) such that $|x-z|\le r+|z|\le c_1\sqrt{1+|z|^2}$ for every $x\in E$. Since
$|y|=\sqrt{|z|^2\phi^2+\phi^2}=|\phi|\sqrt{1+|z|^2}$, therefore  \Cref{lemma:luzin-rot} follows from the trivial estimate \eqref{trivirot}, with $E$ replaced by $R$, and $|x-z|$ replaced by $c_1\sqrt{1+|z|^2}$. 
\end{proof}

\begin{proof}[Proof of \Cref{var2}]
First, suppose that $z \in \R^2$ and $\phi \neq 0$. We note that we can improve the estimate \eqref{trivirot} by noticing that the derivative of $t\mapsto |x(t)-z|$ is $\langle \dot x(t),\frac{x(t)-z}{|x(t)-z|}\rangle = \frac{1}{|x(t)-z|}\dist(\nu_{x(t)},z)$. (Here, $\dist$ denotes the Euclidean distance.) Therefore, by the coarea formula \eqref{eq:coarea-formula}, rotating the set $R$ covers area
\begin{equation}
\label{eq:rotate-R-bound}
\leq |\phi|\int_R \dist(\nu_x,z)\,d \cH^1(x).
\end{equation}

Thus, it suffices to show that if $\nu_x$ intersects the $\delta$-neighbourhood of $z$ in $\P^2$, then $\dist(\nu_x,z) \leq c\delta\sqrt{1+|z|^2}$ in $\R^2$. If $|z| \leq 2r$ and $\delta$ is sufficiently small, then the Euclidean and projective distances are comparable, and $\sqrt{1+|z|^2}$ is comparable to 1 (where the implied constants depend only on $r$), so there is nothing to prove. Now, suppose $|z| > 2r$. Since $E\subset B(0,r)$, therefore, for $\delta$ sufficiently small, the projective ball $B(z, \delta)$ is bounded away from $x\in E$. Let $\pi$ denote the quotient map $\pi:\,S^2\to\P^2$. Then there is a constant $c_2$ (that depends only on $r$) such that the angle between any two great circles through $\pi^{-1}x$ that meet $\pi^{-1}(B(z,\delta))$ is $\le c_2\delta$. Then there is a constant $c_3$ (that depends only on $r$) such that the angle between $x-z$ and $\nu_x$ in $\R^2$ is $\le c_2c_3\delta$. With $c_1$ as in the proof of \Cref{lemma:luzin-rot}, we have $\dist(\nu_x,z)\le c_2c_3\delta|x-z|\le c_1c_2c_3\delta\sqrt{1+|z|^2}$, as desired.

Now, we prove \Cref{var2} when $z \in \P^1$ (i.e., when $\rho$ is a translation). Again, for $\delta$ sufficiently small, the projective ball $B(z,\delta)$ is bounded away from $E$. Thus, if $\nu_x$ intersects $B(z, \delta)$, then the angle between $\nu_x$ and $\P^1$ is bounded away from zero. Therefore, there is a constant $c_4$ (that depends only on $r$) such that if $\nu_x$ intersects $B(z, \delta)$ for some $x \in E$, then the projective distance between $z$ and $\nu_x \cap \P^1$ is $\leq c_4 \delta$. Hence, we can apply \Cref{lemma:variation} to obtain our desired result.
\end{proof}

\section{Kakeya needle problem for translations}\label{sec:translations}

In this section $E$ is an arbitrary rectifiable set of finite $\cH^1$-measure. Without loss of generality we assume that $\cH^1(E)=1$, and that $\theta_x$ is defined for each $x\in E$, as in \Cref{sec:rectifiable}.

\subsection{Notation}

We say that a subset of $\P^1$ is an \emph{interval} if it is connected. For $\theta_1,\theta_2\in \P^1$ with $|\theta_1-\theta_2|<\pi/2$, we denote by $[\theta_1,\theta_2]$ the interval in 
$\P^1$ whose endpoints are $\theta_1,\theta_2$ and has length less than $\pi/2$. (When we use this notation, we do not specify which one is the left and which one is the right endpoint.) %

The symbol $\ii$ will always denote a finite binary sequence, i.e., a sequence $i_1i_2\dots i_k$, where $k \geq 0$ and each term $i_j$ is 0 or 1. The empty sequence $\emptyset$ corresponds to $k=0$. The length of  $\ii$ is denoted by $|\ii|$. We denote $\ii'=i_1i_2\dots i_{k-1}$ (note that $\emptyset'$ is not defined). 
We will say that $\ii'$ is the \emph{parent} of $\ii$, and $\ii$ is a \emph{child} of $\ii'$, respectively. The \emph{ancestors} and the \emph{descendants} of an $\ii$ are defined in the obvious way. We will also say that a sequence is \emph{bad} if it ends with a $0$ and \emph{good} if it ends with a $1$. (The empty sequence $\emptyset$ is also good.)

\subsection{Basic zigzag}
\label{sec:basic-zigzag-lines}

A \emph{basic zigzag} is a polygonal path which is made up of $N$ congruent and equally spaced segments in direction $\theta_0$ interlaced with $N$ congruent segments in direction $\theta_1$. (See \Cref{figure:zigzags}(b) for an example.)

The fundamental procedure in our construction is taking a  line segment $L$ and replacing it with a basic zigzag with the same endpoints. The key properties of basic zigzags are the following two geometrically obvious facts:
\begin{itemize}
  \item With $L, \theta_0, \theta_1$ fixed, we can ensure that the basic zigzag lies in an arbitrarily small neighborhood of $L$ by making the zigzag sufficiently ``fine,'' i.e., making $N$ sufficiently large. 
  \item The total length of each of the two parallel pieces of the basic zigzag depends only on $L, \theta_0, \theta_1$ and not on the fineness of the zigzag.
\end{itemize}

\subsection{Venetian blind}
\label{sec:venetian}

Like the basic zigzag, a Venetian blind is a polygonal path of line segments of two fixed directions. These segments are constructed by iterating the basic zigzag construction. We fix a line segment $L$, small parameters $0 < \gamma \leq \beta < \frac{\pi}{4}$ and a sign $\sigma \in \{-1,1\}$. The Venetian blind construction  is as follows. 

Let $\theta_L\in\P^1$ denote the direction of $L$. In our first step, we replace $L$ by a basic zigzag with directions $\theta_L-\sigma\beta$, $\theta_L+\sigma\gamma$. (See \Cref{figure:zigzags}(b).) Let $G_1$ denote the union of the line segments of the basic zigzag in direction $\theta_L+\sigma\gamma$. Iteratively, in our $i^{th}$ step for $i\ge 2$, we replace each line segment in $G_{i-1}$ by a basic zigzag of directions $\theta_L-\sigma\beta$, $\theta_L+ i\sigma\gamma$, and let $G_i$ denote the union of the line segments in direction $\theta_L+ i\sigma\gamma$. (See \Cref{figure:zigzags}(c) for the line segments obtained after the second step.) We stop this procedure after $k$ steps, where $k$ is defined by
\begin{equation}\label{rat}
k\gamma\in[\pi/2-2\beta-\gamma,\pi/2-2\beta).
\end{equation}

The zigzag we end up with is what we call our \emph{Venetian blind}.
We denote by $L_1$ the final set $G_k$ obtained by this construction, and we denote by $L_0$ the rest of the Venetian blind. That is, the Venetian blind is the polygonal path $L_0\cup L_1$ where $L_0$ and $L_1$ are unions of line segments of directions $\theta_L-\sigma \beta$ and $\theta_L+ k\sigma\gamma$ respectively.
We call $L_0$ the \emph{bad} part of the Venetian blind and $L_1$ the \emph{good} part.

We say that the directions $\theta_L - \sigma\beta, \theta_L, \theta_L + \sigma\gamma, \ldots, \theta_L + k\sigma\gamma$ have been \emph{used} in the construction of this Venetian blind. This terminology will be used in \Cref{subsection:fineness-of-the-zigzags}.

\begin{remark} Usually, in the literature, a Venetian blind consists only of the ``good'' line segments. In our definition of a Venetian blind, it contains both $L_0$ and $L_1$.
\end{remark}

\begin{remark}
\label{remark:Venetian-blind-lengths}
The lengths of $L_0$ and $L_1$ depend only on $\cH^1(L)$ and $\beta,\gamma$. They \emph{do not} depend on the fineness of the zigzags. Furthermore, condition \eqref{rat} ensures that there is a constant $c(\beta)<1$ such that
\begin{equation}\label{ratt}
\cH^1(G_j)\le \cH^1(L),\ \cH^1(L_i)\le c(\beta)\cH^1(L)
\end{equation}
for each $j=1,2,\dots,k$ and $i=0,1$.
\end{remark}

\begin{remark}
\label{remark:Venetian-blind-bad-segments}
We consider two natural ways to partition $L_0$ into line segments. The first way is into the maximal disjoint line segments of $L_0$.  Note that a line segment in this partition could be made up of segments from multiple basic zigzags in our construction. (For example, the right-most segment in \Cref{figure:zigzags}(c) has this property.) 

The second partition is a refinement of the first. We subdivide each segment from the first partition into the individual segments from the basic zigzags, i.e., each segment in the second partition is a segment from some basic zigzag used in the Venetian blind construction.

In \Cref{section:proof-of-translations} (below), we describe how to iterate the Venetian blind construction on each line segment of $L_0$. We can interpret the word ``each'' in two different ways, corresponding to the two decompositions above. In this section, it does not matter which way we choose, but in \Cref{sec:rotations}, we must choose the second one.
\end{remark}

\begin{figure}
  \centering
    \includegraphics[width=0.7\textwidth]{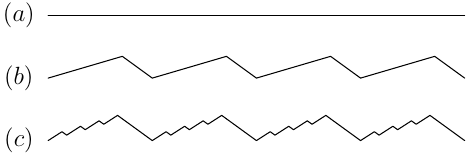}
  \caption{Two steps in the Venetian blind construction, starting with (a) as the initial segment.}
  \label{figure:zigzags}
\end{figure}

\subsection{Main construction}\label{section:proof-of-translations}

Our strategy of proving \Cref{theorem:translations} is to iterate the Venetian blind construction. Given a point in $\R^2$, we construct a polygonal path from the origin to this point. We start with the line segment $L_\emptyset$ joining the origin to this point, and then, iteratively, for each finite sequence $\ii$, we apply the Venetian blind construction to each segment in $L_{\ii}$ with some parameters $\beta = \beta_\ii,\gamma=\gamma_\ii, \sigma = \sigma_\ii$. We let $L_{\ii 0}$ be the union of all the bad parts of the Venetian blinds and $L_{\ii 1}$ be the union of all the good parts (as defined in \Cref{sec:venetian}). Since we use the same $\beta_\ii, \gamma_\ii, \sigma_\ii$ on each line segment in $L_{\ii}$, it follows by induction that every $L_\ii$ is a union of parallel line segments of some direction $\theta_\ii$.

We also iteratively assign, to each ${\bf i}$, an interval $I_\ii\subset\P^1$ by the following simple method. We put $I_\emptyset=\emptyset$. Then, for each finite sequence $\ii\neq\emp$, we define $I_{\ii}:=I_{\ii'}\cup[\theta_\ii,\theta_{\ii'}]$. Then clearly, by induction, we can see that for every $\ii\neq\emptyset$, $I_{\ii}$ is an interval and $\theta_{\ii}\in I_\ii$.

\subsection{Choosing the parameters $\beta_\ii,\gamma_\ii$}\label{sec:choose}

For each $\ii$, we fix a small $\eps_\ii$ that we will specify later. They will depend only on $\eps$ (where $\eps$ is from the statement of \Cref{theorem:translations}). We denote the number of 1's in the sequence $\ii$ by $n_\ii$. 
Then we can choose our parameters $\beta_\ii\ge \gamma_\ii>0$ in our Venetian blind constructions such that they satisfy:
\begin{enumerate}
\item $\beta_{\ii}\le\beta_{\ii'}$ for every $\ii$;
\item $\beta_{\ii}=\beta_{\kk}$, where $\kk$ is the last (i.e., youngest) good sequence among $\ii$ and its ancestors;
\item $\beta_\ii\le 1/n_\ii$ for every $\ii$;
\item $\beta_{\ii}\cH^1(L_{\ii})\le \eps_{\ii}$ if $\ii$ is good;
\item $\gamma_\ii\cH^1(L_{\ii})\le \eps_{\ii}$ for every $\ii$. 
\end{enumerate}
We can indeed make these choices, since $\cH^1(L_{\ii})$ is determined by the $\beta$'s and $\gamma$'s of its ancestors. 

We will also use the notation: 
\begin{enumerate}
\item[(6)] $\alpha_\ii=\beta_{\ii'}$ if $\ii$ is bad, and $\alpha_\ii=\gamma_{\ii'}$ if $\ii$ is good.
\end{enumerate}

\subsection{Choosing the signs $\sigma_\ii$} 
We choose each sign $\sigma_\ii$ such that it makes $I_{\ii 1} = I_\ii \cup [\theta_\ii, \theta_{\ii 1}]$ as large as possible. That is, if $\theta_\ii$ is in the right half of the interval $I_\ii$ (where we embed $I_\ii$ into $\R$), then we choose $\sigma_\ii = 1$; otherwise, we choose the $\sigma_\ii = -1$. (If $\ii=\emptyset$, or if $\theta_\ii$ is in the middle of the interval $I_\ii$, or if $I_\ii=\P^1$, then we can choose the sign arbitrarily.) Our choice of $\sigma_\ii$ ensures that the length of the interval $I_{\ii 1}$ can be estimated by
$$|I_{\ii 1}|\ge |I_{\ii}|/2+|\theta_\ii-\theta_{\ii 1}|\ge |I_\ii|/2+\pi/2-2\beta_\ii.$$
The second inequality follows from \eqref{rat}. This can be re-written as:
$$\pi-|I_{\ii 1}|\le(\pi-|I_\ii|)/2+2\beta_\ii.$$
Suppose $\kk$ is the last good sequence among $\ii$ and its ancestors, and ${\mm}$ is the second-to-last one. Then since the intervals $I_\ii$ are increasing and the parameters $\beta_\ii$ are decreasing along each family line, we have $\pi-|I_{\ii}|\le \pi-|I_{\kk}|$ and hence
\begin{equation}\label{I}
\pi-|I_{\ii}|\le(\pi-|I_{\mm}|)/2+2\beta_{\mm}.
\end{equation} 

\subsection{Stopping time}\label{sec:stopping-time}
We need to define when we stop our Venetian blind constructions on various family lines. In order to construct our polygonal path $P$ (for \Cref{theorem:translations}), we need to ensure that %
ultimate extinction occurs.
 We will, of course, define the polygonal path $P$ as the union of those $L_\ii$ where the construction stops.

First of all, we stop our Venetian blind construction at $L_\ii$ if $\cH^1(L_{\ii})\le\eps_\kk$, where $\kk$ is the last good sequence among $\ii$ and its ancestors. By \eqref{ratt} and condition (2) in \Cref{sec:choose},
\begin{equation}\cH^1(L_\ii)\ge c(\beta_\ii)^{-1}\cH^1(L_{\ii 0})\ge c(\beta_\ii)^{-2}\cH^1(L_{\ii 00})\ge\dots
\end{equation} 
This ensures that, for each $\ii$, the family line $\ii,{\ii 0},{\ii 00},\dots$ dies out after finitely many generations, where the number of generations depends only on $\beta_\ii$. Therefore $\min\{n_\ii:\,|\ii|=k\}\to\infty$ as $k\to\infty$, and consequently, by condition (3), $\max\{\beta_\ii:\,|\ii|=k\}\to 0$. Using this and \eqref{I}, if $k$ is large enough, then 
\begin{equation}\label{k}
\max\{\pi-|I_\ii|:\,|\ii|=k\}<\eps.
\end{equation}
We stop our whole construction after $k$ generations, where $k$ is so large that \eqref{k} holds.

By choosing the parameters $\eps_\ii$ such that $\sum_\ii\eps_\ii$ is small enough, by \Cref{lemma:luzin}, \Cref{remark:luzinrem} and our assumption $\cH^1(E)=1$, we can ignore those $L_\ii$ for which  $\cH^1(L_{\ii})\le\eps_\kk$, where $\kk$ is the last good sequence among $\ii$ and its ancestors. (This is because each $\ii$ at which we stop our construction has a different ``last good sequence among $\ii$ and its ancestors.'') For the line segments that belong to the remaining part of the polygonal path, we have $|I_\ii|>\pi-\eps$ by \eqref{k}. 

Using the previous paragraph, we choose our balls $B(\theta_i, \eps)$ for \Cref{theorem:translations} as follows. If we ignore $L_\ii$ (as described in the previous paragraph), then for each line segment $L_i$ in $L_{\ii}$, we let $\theta_i$ to be any point we like.  If we do not ignore $L_\ii$, then for each $L_i$ in $L_\ii$, we choose $\theta_i$ so that $\P^1 \setminus B(\theta_i, \eps) \subset  I_\ii$. We can do this because $|I_\ii|>\pi-\eps$. In both cases, since $\theta_\ii \in I_\ii$ for all $\ii$, we can also choose each $\theta_i$ so that 
\begin{equation}\label{thetai}
\theta_\ii \not\in B(\theta_i, \epsilon).
\end{equation}

In order to finish the proof of \Cref{theorem:translations}, it suffices to show that
\begin{equation}\label{4.6}
A:=\bigcup_{L_\ii\subset P} (L_\ii+\{x\in E:\,\theta_x\in I_\ii\})
\end{equation} 
has small measure.

\begin{remark}
So far, our definition of the polygonal path $P$ did not depend on the set $E$. In what follows, we will show that if the zigzags we use are sufficiently fine (depending on the set $E$) then indeed the set $A$ in \eqref{4.6} has small measure. Note that the fineness of the zigzags is the only remaining parameter we need to specify. The parameters $\beta_\ii, \gamma_\ii, \sigma_\ii$,  the lengths $\cH^1(L_\ii)$, the stopping time, and the intervals $I_\ii$ are all independent of the fineness of the zigzags and of $E$.
\end{remark}

\subsection{Fineness of the zigzags, and the small neighborhood lemma} 
\label{subsection:fineness-of-the-zigzags}

We have already chosen all the directions we use in all the basic zigzags to construct $P$. These directions divide $\P^1$ into finitely many intervals, which we call \emph{elementary intervals}. By an elementary interval we mean a \emph{closed} interval $I\subset\P^1$ such that its endpoints are directions used in our construction, and such that $I$ does not contain any other such direction. (See \Cref{sec:venetian} for what it means for a direction to be ``used in our construction.'')

Since we already know the length $\cH^1(P)$, we also know how large subset of $E$ we may ignore by \Cref{lemma:luzin} and \Cref{remark:luzinrem}. Therefore, by throwing away a sufficiently small subset of $E$ if necessary, we can assume that $E$ is compact and also that $x\mapsto\theta_x$ is a continuous function on $E$. 

For an elementary interval $I$, we denote
$$E_I=\{x\in E:\,\theta_x\in I\}.$$
Because of our assumptions above, $E_I$ is also compact.

Here is our strategy for choosing the fineness of the zigzags. Suppose that for some $E_I$ and for some line  segment $L$ in our construction, we have obtained the estimate $|L + E_I| < \eta$ for some $\eta$. Then, we require all zigzags descending from $L$ to be fine enough so that they stay in a sufficiently small neighborhood of $L$. This ensures that by the small neighborhood lemma, translating $E_I$ along the descendants of $L$ still covers area $< \eta$. %

In the next section, we obtain finitely many estimates of the form $|L + E_I| < \eta$. We make the zigzags sufficiently fine at each step so that these estimates are preserved by the descendants of $L$, as explained above.

\subsection{Area estimate}
We fix an elementary interval $I$ and the corresponding set $E_I$, and revisit the Venetian blind construction. Our aim is to estimate the measure of the set
\begin{equation}\label{4.7}
A_I:=\bigcup_{L_\ii\subset P \text{ s.t. } I \subset I_\ii} (L_\ii+E_I).
\end{equation} 
Our final goal is to show $|\bigcup_I A_I| < \epsilon$.
In the next two paragraphs, we will use the same notations as in \Cref{sec:venetian}.

First assume that the elementary interval $I$ is contained in the interval $[\theta_L + (j-1)\sigma\gamma$, $\theta_L+ j\sigma\gamma]$ for some $j=1,2,\dots,k$. Since the line segments of $G_j$ are of direction $\theta_L+ j\sigma\gamma$, it follows from \Cref{lemma:variation} (and the estimate \eqref{ratt}) that translating $E_I$ along the line segments of $G_j$ covers area $\lesssim\gamma\cH^1(E_I)\cH^1(G_j)
\le\gamma\cH^1(E_I)\cH^1(L)$. By our remarks in the previous section about choosing the fineness of the zigzags, the same estimate $\gamma \cH^1(E_I)\cH^1(L)$ remains true if we translate $E_I$ along the line segments of $G_k$. 

We can argue similarly when the elementary interval is contained in $[\theta_L,\theta_L-\sigma\beta]$. Therefore, we proved the following lemma:
\begin{lemma}\label{lemma:lemm}
  Suppose that an elementary interval $I$ is contained in $[\theta_\ii,\theta_{\ii'}]$. Then $|L_\ii + E_I|\lesssim\alpha_\ii \cH^1(E_I)\cH^1(L_{\ii'})$.
\end{lemma}

Now consider an $L_\ii\subset P$ with $I\subset I_\ii$. Since the intervals $I_\ii,I_{\ii'},I_{\ii''},\dots$ are decreasing, there is a $\kk$ among $\ii$ and its ancestors such that $I\subset I_\kk\setminus I_{\kk'}\subset[\theta_\kk,\theta_{\kk'}]$. By \Cref{lemma:lemm}, 
\begin{equation}\label{l''}
|L_{\kk} + E_I| \lesssim \alpha_\kk \cH^1(E_I)\cH^1(L_{\kk'}).
\end{equation}

The estimates \eqref{l''} are precisely those that we would like to maintain when we replace the set $L_\kk$ by all of its final descendants in $P$, as described in the previous section. Thus, we make the zigzags sufficiently fine so that these estimates are preserved.

Therefore, instead of taking the sum of the estimates \eqref{l''} for \emph{all} finite sequences $\kk$, it is sufficient to take the sum for \emph{some} $\kk$, each of which belongs to a different family line. Let $\kk_1,\kk_2,\dots$ be arbitrary sequences from different family lines.

We distinguish two cases: if $\kk_m$ is good, then by (5) and (6) in \Cref{sec:choose}, 
\begin{equation}\label{1}
\alpha_{\kk_m}\cH^1(E_I)\cH^1(L_{{\kk_m}'})
\le\eps_{{\kk_m}'}\cH^1(E_I).
\end{equation}
With the bad $\kk_m$, the same trivial bound does not work. Nonetheless, because of the ``different family lines condition,'' each bad $\kk_m$ has a different ``last good among $\kk_m$ and its ancestors.'' Therefore by (2), (4), and (6) in \Cref{sec:choose} and \eqref{ratt}, we have
\begin{equation}\label{2}
\sum_{\kk_m \text{ bad}} \alpha_{\kk_m}\cH^1(E_I)\cH^1(L_{{\kk_m}'})
\le\sum_{\kk}\eps_{\kk}\cH^1(E_I),
\end{equation}
where the summation on the right is taken over all ${\kk}$. Adding together the estimates \eqref{1} for all good $\kk_m$ and \eqref{2}, we have
\begin{equation}
\label{eq:AI-bound}
|A_I|\le 2\sum_{\kk}\eps_{\kk}\cH^1(E_I).\end{equation}
Since each $x\in E$ belongs to at most two of the sets $E_I$, by summing over $I$ and choosing $\sum_{\kk}\eps_{\kk}$ small enough, the proof of \Cref{theorem:translations} is finished.

\section{Kakeya needle problem for rotations}\label{sec:rotations}

Our aim in this section is to prove the following theorem, which can be thought of as a direct analogue of \Cref{theorem:translations}. Recall when we create a polygonal path $P \subset \Isom^+(\R^2)$ from a sequence of rotations $\{\rho_i\}$, we always interpret the rotations in the \emph{intrinsic} coordinate system. We will occasionally use the phrase \emph{intrinsic rotation} to remind ourselves of this convention.

\begin{theorem}
\label{theorem:rotations-on-line}
Let $E \subset \R^2$ be a bounded rectifiable set of finite $\cH^1$-measure. Let $\epsilon > 0$, and let $\rho\in\Isom^+(\R^2)$ be arbitrary. Let $\ell\subset\P^2$ be a line through the projective center $z$ of $\rho$.

Then there are intrinsic rotations $\rho_i=\rho(x_i)$ with projective centers $z_i \in B(\ell, \epsilon) \subset \P^2$ such that the corresponding polygonal path $P=\bigcup_i L_i\subset\Isom^+(\R^2)$ connects the identity and $\rho$, and for each $i$, there exists a $u_i \in \ell$ such that 
\begin{equation}
\label{eq:theorem-rotations}
|\bigcup_{i}\bigcup_{p\in L_i} p(\{x\in E: \nu_x \cap \ell \cap B(u_i, \eps)  = \emptyset\})|< \eps.
\end{equation}
\end{theorem}

\subsection{Basic zigzags, deconstructed}
\label{sec:basic-zigzag}
The heart of the matter in our proof of \Cref{theorem:translations} was that we repeatedly replaced line segments by basic zigzags. Each line segment $L$ represented a translation. 
In our proof of \Cref{theorem:rotations-on-line} we will do an analogue construction with rotations instead of translations. However, this is a bit more delicate, so first, we present the basic zigzag construction for translations in more detail than before. We decompose this construction into two steps. 

The first step of the basic zigzag construction for translations divides a line segment $L$ into $N$ equal parts. In the second step, for translations, we replace each of the $N$ line segments by two line segments of given directions. We can represent these two steps by the two equations
\begin{align*}
v &= (v/N) + \cdots + (v/N)
\\
&= (v_0/N) + (v_1/N) + \cdots + (v_0/N) + (v_1/N),
\end{align*}
where $v_0$ and $v_1$ are vectors in the two given directions and such that $v = v_0 + v_1$.

The first step for rotations is easy to understand: we replace a rotation $\rho=\rho(x)$ by $N$ copies of $\rho(x/N)$, which are rotations around the same projective center as $\rho$ but with angle reduced by a factor of $N$. In the intrinsic coordinate system, if we apply $\rho(x/N)$ repeatedly $N$ times, then indeed we obtain $\rho(x)$. 

The second step for rotations would be to replace each $\rho(x/N)$ by $\rho(x_0/N)$ and $\rho(x_1/N)$, for some $x_0$ and $x_1$. We need to determine the necessary condition on $x,x_0, x_1$, i.e., the analogue of $v = v_0 + v_1$. It is not as simple as $x = x_0 + x_1$; the composition of $\rho(x_0)$ followed by $\rho(x_1)$ is not necessarily $\rho(x_0+x_1)$. Therefore, first we need to understand which rotations a given $\rho$ can be replaced by. We do this in the next section.

\subsection{The structure of intrinsic compositions}

Using the notation $\rho=\rho(w,\phi)$ and $v$ from \Cref{rott}, we can see that $\rho_3$ can be replaced by $\rho_1,\rho_2$ if
\begin{equation}\label{eq:phi123}
\phi_1+\phi_2=\phi_3
\end{equation}  
and 
\begin{equation}\label{v12}
v_1+e^{i\phi_1}v_2=v_3.
\end{equation} 
Indeed, \eqref{eq:phi123} says that by applying $\rho_1$ and $\rho_2$, we rotate $\R^2$ by angle $\phi_1+\phi_2$. And \eqref{v12} says that the image of $0$ after applying $\rho_1$ and $\rho_2$ will be $v_1+e^{i\phi_1}v_2$. To see this,  the first rotation, $\rho_1$, displaces $0$ by $v_1$. Then, $\rho_2$ displaces it further. This displacement is $v_2$ in the intrinsic coordinate system and $e^{i\phi_1}v_2$ in the extrinsic coordinate system, where the extra factor of $e^{i\phi_1}$ is due to the offset in directions between the intrinsic and extrinsic coordinate systems introduced by $\rho_1$. If two rotations have the same angle and they map 0 to the same point, then they are the same rotation. 

For $x_j\in\R^3\setminus\{0\}$, we will use the notation $x_3=x_1\star x_2$ if \eqref{eq:phi123} and \eqref{v12} hold for $\rho_j=\rho(x_j)$. %

\begin{remark}
Note that \eqref{eq:phi123} and \eqref{v12} hold if and only if $\rho_3 = \rho_1 \circ \rho_2$. In general, the composition of two \emph{intrinsic} rotations $\rho_1$ and $\rho_2$ (in that order) is $\rho_1 \circ \rho_2$, \emph{not} $\rho_2 \circ \rho_1$.
\end{remark}

\begin{remark}
We do not need the following fact in this paper, but the conditions \eqref{eq:phi123} and \eqref{v12} imply that $\star$ is a group operation on $\R^3$. The group $(\R^3, \star)$ has the structure of the semidirect product $\R^2 \rtimes \R$, where $\phi \in \R$ acts on $v \in \R^2$ by $v \mapsto e^{i\phi} v$. 

The extra difficulty in our proof for rotations is essentially due to the failure of $\star$ to agree with $+$. Nonetheless, we can modify the proof for translations to obtain a proof for rotations because for small $x_1, x_2$, $\star$ is ``close enough'' to $+$, as we show in the next section. 
\end{remark}

Our main estimate is the following:
\begin{lemma}\label{lemma:star-approx-add}
Let $x_j=(w_j,\phi_j)\in \R^3\setminus\{0\}$ with $x_3=x_1\star x_2$ and $|\phi_j|\lesssim 1$ for each $j$. Then 
\begin{equation}\label{eq:star-approx-add}
|w_1 + w_2 - w_3| \lesssim |w_2\phi_1| + |w_1\phi_1| + |w_2\phi_2| + |w_3\phi_3|.
\end{equation}
\end{lemma}

\begin{proof}
Observe that $|v_j| \leq |w_j|$ and $|v_j + iw_j| \lesssim |z_j \phi_j^2| = |w_j \phi_j|$. (For the second inequality, we used $|\phi_j| \lesssim 1$.) By \eqref{v12}, $|v_1 + v_2 - v_3| = |v_2(1-e^{i\phi_1})| \leq |w_2 \phi_1|$. Thus indeed,
\begin{align*}
|w_1 + w_2 - w_3|
&\leq
|v_1 + v_2 - v_3| + |v_1 + iw_1| + |v_2 + iw_2| + |v_3 + iw_3|
\\
&\lesssim
|w_2 \phi_1| + |w_1 \phi_1| + |w_2 \phi_2| + |w_3\phi_3|.
\qedhere
\end{align*}
\end{proof}

\subsection{Basic zigzag construction for rotations}
Now we are ready to define our basic zigzag construction in general. This construction, for given $x,x_0,x_1\in\R^3\setminus\{0\}$ with $x=x_0+x_1$ and a given $N$, replaces the rotation $\rho(x)$ by the sequence of intrinsic rotations $\rho(y_0),\rho(y_1),\dots,\rho(y_0),\rho(y_1)$. We define $y_0=x_0/N$, and then $y_1=\tilde x_1/N$ is defined by $y_0 \star y_1 = x/N$.

The key properties of the construction are the following. 

\begin{lemma}
\label{lemma:tilde-x1}
For any given $\eps>0$, if $N$ is sufficiently large, then:
\begin{enumerate}
\item $|y_j|<\eps$ for $j = 0, 1$;
\item $|\tilde x_1-x_1|<\eps$.
\end{enumerate}
\end{lemma}

\begin{proof}
Since $y_0 = x_0/N$, property $(1)$ for $j = 0$ is obvious. For $j = 1$, this property follows from $y_1 = \tilde x_1/N$ and from (2).

Let $x = (w, \phi)$, $x_j = (w_j, \phi_j)$, and $\tilde x_1 = (\tilde w_1, \tilde \phi_1)$. To prove (2), it suffices to show that $\tilde w_1\to w_1$ as $N\to\infty$, since $\tilde\phi_1=\phi_1$. If $N$ is large enough, then $\frac{\phi}{N},\frac{\phi_j}{N}\lesssim 1$, so we can apply \Cref{lemma:star-approx-add} for $x/N= (x_0/N) \star (\tilde x_1/N)$ to obtain:
$$|\tilde w_1/N+w_0/N-w/N|\lesssim
\frac{1}{N^2}(|w_0\phi_1|+|\tilde w_1\phi_1|+|w_0\phi_0|+|w\phi|).$$
Therefore 
\begin{align*}
|\tilde w_1-w_1|=|\tilde w_1+w_0-w|
&\lesssim \frac{1}{N}(|w_0\phi_1|+|\tilde w_1\phi_1|+|w_0\phi_0|+|w\phi|)
\\
&\leq
\frac{1}{N}(c_1 + c_2 |\tilde w_1|)
\end{align*}
for some $c_1, c_2$ independent of $N$ (since $w, w_0, w_1, \phi, \phi_0, \phi_1$ do not depend on $N$). Therefore $\tilde w_1\to w_1$ as $N\to\infty$.
\end{proof}

Property (1) allows the polygonal path for $y_0 \star y_1 \star \cdots \star y_0 \star y_1$ to stay within an arbitrarily small neighborhood of the line segment defined by $x$. This is because by decomposing  $x$ into $(x/N)\star\dots\star(x/N)$, we divide the line segment into $N$ equal segments. When we replace each segment by $y_0\star y_1$, we stay in a small neighborhood of it.

\subsection{Iterating the basic zigzag}
\label{sec:iterating-basic-zigzag}

In our proof of \Cref{theorem:translations}, we started from a line segment $L$ and then, iteratively, we replaced each line segment by a Venetian blind; the indices $\ii$ indexed the Venetian blinds. %
However, in this section, we need to focus also on basic zigzags, hence we introduce a new set of indices $\jj$ (finite binary sequences) to index the basic zigzags. For $\jj = j_1 \cdots j_k$, we denote $\jj' = j_1 \cdots j_{k-1}$.

Our construction from \Cref{sec:translations} is an iteration of basic zigzag constructions. That is, we begin with a line segment and replace it a basic zigzag. Then we iterate this by replacing each line segment of our basic zigzag with a basic zigzag. (For this to be an accurate description of our construction from \Cref{sec:translations}, we must use the ``second partition'' from \Cref{remark:Venetian-blind-bad-segments}.)

Given such an iteration of basic zigzags, we can describe it as follows. We start with a line segment $L_\emptyset$, which corresponds to a translation by a vector $v_\emptyset \in \R^2$. In our first basic zigzag, we chose two directions $\theta_0$, $\theta_1$. Then we can uniquely decompose $v_\emptyset = v_0 + v_1$, where $v_j$ is in direction $\theta_j$. If the fineness is $N$, we can represent the basic zigzag as \[v_\emptyset = (v_0/N) + (v_1/N) + \cdots + (v_0/N) + (v_1/N).\] This gives us $N$ copies of the segments $v_0/N$ and $v_1/N$. We set $M_0 = M_1 = N$.

Now suppose we have $M_\jj$ copies of $v_\jj/M_\jj$. To apply a basic zigzag on every copy, we write $v_\jj=v_{\jj 0}+v_{\jj 1}$ and choose a fineness $N_\jj$. Then our basic zigzag is
\[
\frac{v_\jj}{M_\jj}
=
\frac{v_{\jj 0}}{M_\jj N_\jj}+\frac{v_{\jj 1}}{M_\jj N_\jj}
+
\cdots
+
\frac{v_{\jj 0}}{M_\jj N_\jj}+\frac{v_{\jj 1}}{M_\jj N_\jj}.
\]
Here we have $N_\jj$ copies of $\frac{v_{\jj 0}}{M_\jj N_\jj}+\frac{v_{\jj 1}}{M_\jj N_\jj}$ and $M_{\jj 0}=M_{\jj 1}=M_\jj N_\jj$.

We let $L_\jj \subset \R^2$ be the union of the $M_\jj$ congruent and parallel line segments corresponding to the $M_\jj$ copies of $v_\jj/M_\jj$. We let $\theta_\jj$ be the direction of these segments.

\begin{remark}
The vectors $v_\jj$ do not depend on the fineness of the zigzags. Note also that $|v_\jj| = \cH^1(L_\jj)$. 
\end{remark}

\begin{remark}\label{belonging}
As noted earlier, the indices $\jj$ index the  basic zigzag constructions from \Cref{sec:basic-zigzag-lines}, whereas the indices $\ii$ index the Venetian blind constructions from \Cref{sec:venetian}. Note that each $L_\ii$ is a union of $L_\jj$ over some set of indices $\jj$.

Fix some $\ii$ and $\jj$ with $L_\jj \subset L_\ii$. When we apply the Venetian blind construction to $L_{\ii}$, we obtain $L_{\ii 0}$ and $L_{\ii 1}$.  As part of this procedure, we iterate the basic zigzag construction $k$ times on $L_\jj$, where in the $i$th step ($1 \leq i \leq k$), we replace $L_{\jj1^{i-1}}$ with the basic zigzags $L_{\jj1^{i-1}0} \cup L_{\jj1^{i-1}1}$. (Here, $1^i$ denotes a string of $i$ $1$s.) In the end, we obtain
\[
\bigcup_{i=0}^{k-1}
 L_{\jj1^i0} \cup L_{\jj1^k},
\quad\text{with}\quad
\bigcup_{i=0}^{k-1}
 L_{\jj1^i0}
  \subset L_{\ii 0}
\quad\text{and}\quad 
L_{\jj1^k}
\subset L_{\ii 1}
.
\]

For each $i = 0, \ldots, k-1$ we say that the index $\jj1^i0$ is \emph{between $\ii$ and $\ii 0$} and that the index $\jj1^i1$ is \emph{between $\ii$ and $\ii 1$}. (Note that by this definition, if $\ii \neq \emptyset$, then $\jj$ is between $\ii'$ and $\ii$.)
\end{remark}

\medskip

In the paragraphs above, we showed how to construct $\{v_\jj\}$ given an iteration of basic zigzags. Conversely, we could start with a collection $\{v_\jj\}$ satisfying $v_\jj = v_{\jj 0} + v_{\jj 1}$ and turn this into instructions for iterating the basic zigzags. (We would also need to specify the fineness $N_\jj$ at each step.)

The analogue of the above scheme for rotations is the following. Suppose that we are given some points $x_{\bf j}\in\R^3\setminus\{0\}$, where the $\jj$ are finite binary sequences, such that 
$x_{\jj}=x_{\jj 0}+x_{\jj 1}$ for each $\jj$. We also fix a small $r>0$.

In our first step of the construction, we choose a sufficiently large $N$ and choose $y_0$, $y_1$ as in the previous section. That is, we replace $x$ by $N$ copies of $(x_0/N)\star(\tilde x_1/N)$:
$$x=(x_0/N)\star(\tilde x_1/N)\star\dots\star(x_0/N)\star(\tilde x_1/N).$$
We choose $N$ so large that $|\tilde x_1-x_1|<r$. (We can do this by \Cref{lemma:tilde-x1}(2).) We also put $\tilde x=x$, $\tilde x_0=x_0$, and $M_0=M_1=N$. 

Now suppose that we have already chosen $\tilde x_\jj$ and an $M_\jj$ for some sequence $\jj$, and $|\tilde x_\jj-x_\jj|<r$. Then we apply a basic zigzag construction with $x$ replaced by
$\tilde x_\jj/M_\jj$, $x_0$ replaced by $x_{\jj 0}/M_{\jj}$ and $x_1$ replaced by
$(x_{\jj 1}+\tilde x_\jj-x_\jj)/M_\jj$, and with fineness $N_\jj$. That is, we replace $\tilde x_\jj/M_\jj N_\jj=(x_{\jj 0}/M_\jj N_\jj)\star(\tilde x_{\jj 1}/M_\jj N_\jj)$, giving us
$$\frac{\tilde x_\jj}{M_\jj}
=
\left(\frac{x_{\jj 0}}{M_\jj N_\jj}\right)\star\left(\frac{\tilde x_{\jj 1}}{M_\jj N_\jj}\right)
\star\dots\star 
\left(\frac{x_{\jj 0}}{M_\jj N_\jj}\right)\star\left(\frac{\tilde x_{\jj 1}}{M_\jj N_\jj}\right).$$
If $N_\jj$ is very large, then $\tilde x_{\jj 1}/M_\jj$ will be very close to $(x_{\jj 1}+\tilde x_\jj-x_\jj)/M_\jj$, which means that $\tilde x_{\jj 1}$ will be very close to $x_{\jj 1}+\tilde x_\jj-x_\jj$. Therefore, by choosing $N_\jj$ large enough, $|\tilde x_{\jj 1}-x_{\jj 1}|<r$ holds. We put $\tilde x_{\jj 0}=x_{\jj 0}$ and $M_{\jj 0}=M_{\jj 1}=M_\jj N_\jj$.

Using this procedure, we obtain an $\tilde x_\jj$ and an $M_\jj$ for each $\jj$, such that $|\tilde x_\jj-x_\jj|<r$, and for $y_\jj:=\tilde x_\jj/M_\jj$:
$$y_\jj=y_{\jj 0}\star y_{\jj 1}\star\dots\star y_{\jj 0}\star y_{\jj 1}$$
(where we have $N_\jj$ copies of $y_{\jj 0}\star y_{\jj 1}$).

In this way, we have shown how to take a collection $\{x_\jj\}$ with $x_\jj = x_{\jj 0} + x_{\jj 1}$, together with fineness $N_\jj$, and turn this data into a sequence of rotations, the composition of which is the original rotation $\rho(x)$.

\begin{remark}
For translations, the sequence $\{v_\jj\}$ tells us every direction we will translate in, even before the fineness $N_\jj$ are chosen. However, for rotations, the sequence $\{x_\jj\}$ alone does not tell us the projective centers of the rotations we will use. The centers are given by $\{\tilde x_\jj\}$, which depend on $N_\jj$. The $N_\jj$ in turn depend on $\{x_\jj\}$ and $r$ (in the way explained above) as well as on the area estimates in the following sections.
\end{remark}

\subsection{Turning the translations into rotations}
\label{sec:translations-into-rotations}

In the previous section, we showed how to turn a collection $\{x_\jj\}$ into a sequence of rotations, but we did not say which sequence $\{x_\jj\}$ to start with. We specify that now. The construction of $\{x_\jj\}$ is actually very simple: we use a rotation in $\R^3$ to ``transform'' a sequence of vectors $\{v_\jj\}$ in $\R^2$  into our desired sequence $\{x_\jj\}$.

Let $\rho(x)$ and $\eps$ be as in the statement of \Cref{theorem:rotations-on-line}.  
Then we can apply the results of \Cref{sec:second-key-idea-rot} to $E$; let $c$ be the constant in \Cref{lemma:luzin-rot} and \Cref{var2}. Without loss of generality we can assume that $\eps$ is small enough, so that the conclusion of \Cref{var2} holds for every $\delta<\eps$. 

Let $v \in \R^2$ be an arbitrary vector with $|v|=|x|$. 
We can follow the steps in \Cref{sec:translations} to construct the vectors $v_\jj$ with $v_\emptyset = v$ as well as the stopping time.

Our aim is to ``turn'' the sequence $\{v_\jj\}$ into a sequence of rotations. Let $Q:\R^3 \to \R^3$ be a linear rotation that maps $(v,0)$ to $x$, and that maps the plane $\phi=0$ (i.e., those $x=(w,\phi) \in  \R^3$ for which $\rho(x)$ is a translation) onto the plane of $\ell$ (i.e., those $x=(w,\phi) \in \R^3$ for which the projective center of $\rho(x)$ lies in $\ell$). We define $x_\jj:=Q(v_\jj,0)$ for each $\jj$. Since $Q$ is linear, we do indeed have 
$x_\jj=x_{\jj 0}+x_{\jj 1}$.

We denote by $z_\jj$ the projective image of $x_\jj$ onto $\P^2$. Then $z_\jj\in\ell$. A trivial but very important property we have is this: since $Q$ is an isometry, the distance between any two $z_\jj$ is the same as the angle between the corresponding vectors $v_\jj$. If $\jj$ is between $\ii'$ and $\ii$ (see \Cref{belonging}), we denote $\alpha_\jj:= \alpha_\ii$ and let 
\begin{equation}
\label{eq:def-B-jj}
B_\jj = B(z_\jj, 2\alpha_\jj) \subset \P^2. 
\end{equation}

\begin{remark}
\label{remark:ball-contains-interval}
Suppose $\jj$ is between $\ii$ and $\ii'$. If $\ii$ is good, then the ball $B_\jj$ contains $[z_\jj,z_{\jj'}]\subset \ell$, which is the image of $[\theta_\jj, \theta_{\jj'}] \subset \P^1$ under the rotation $Q$. If $\ii$ is bad, then $B_\jj$ contains $[z_\jj, z_{\ii'}] = [z_\ii, z_{\ii'}] \subset \ell$.
\end{remark}

So far, none of the objects we defined depend on the fineness of the zigzags; they depend only on $E$, $\ell$, $\rho(x)$ and $\eps$.

Now we use the basic zigzag iteration process in \Cref{sec:iterating-basic-zigzag} to obtain $\{\tilde x_\jj\}$ with $N_\jj$ large enough (that we will specify in the next section). We denote the projective center of the rotations by $\tilde z_\jj$. That is, $\tilde z_\jj$ is the image of $y_\jj$ (which is the same as the image of $\tilde x_\jj$) under the projection $\R^3\setminus\{0\}\to\P^2$. We will also denote $x_\jj=(w_\jj,\phi_\jj)$ and $\tilde x_\jj=(\tilde w_\jj,\phi_\jj)$. (Caution: we do not use the notation $v_\jj$ as in \eqref{v}. Instead, the $v_\jj$ satisfy $x_\jj = Q(v_\jj, 0)$.)

\medskip
Recall from \Cref{sec:iterating-basic-zigzag} that $\tilde x_\jj \in B(x_\jj, r) \subset \R^3$, where we can choose $r$ as small as we wish. We choose $r$ small enough so that $r \leq \frac{1}{2} \min_{\jj} |x_\jj|$ and so that for each $\jj$, the image of $B(x_\jj,r)$ under the projection $\R^3\setminus\{0\}\to\P^2$ is contained in $B_\jj \cap B(\ell, \epsilon)$. It follows that $|\tilde x_\jj| \lesssim |x_\jj|$ and $\tilde z_\jj \in B(z_\jj,2\alpha_\jj) \cap B(\ell, \epsilon)$ for each $\jj$.

In the end, we have two polygonal paths. One is  $P = \bigcup_\jj L_\jj \subset \R^2$, corresponding to $\{v_\jj\}$; the other is $\tilde P = \bigcup_\jj \tilde L_\jj \subset \Isom^+(\R^2)$, corresponding to $\{x_\jj\}$. In both cases, we use the same fineness $N_\jj$ (still to be specified). (We also have the same stopping time since that is encoded in the sequences $\{v_\jj\}$, $\{x_\jj\}$.)

Thus, $Q$ ``transforms'' a polygonal path $P = \bigcup_\jj L_\jj \subset \R^2$ into a polygonal path $\tilde P = \bigcup_\jj \tilde L_\jj \subset \Isom^+(\R^2)$ by ``transforming'' $L_\jj$ into $\tilde L_\jj$. Our next aim is to turn the estimates for $P$ we obtained in \Cref{sec:translations} into estimates for $\tilde P$.

\subsection{Ignoring small parts of $E$ and of $\tilde P$}

Recall the definition of the intervals $I_\ii$, the elementary intervals $I$, and the sets $E_I$ from \Cref{sec:translations}. Because of the rotation $Q$, the relevant objects are now $J_\ii:=QI_\ii$, $J:=QI\subset \ell$, and $E_J:=\{x\in E:\,\nu_x \cap J \neq \emptyset\}$.

We made the sets $E_I$ compact by ``ignoring'' a sufficiently small subset of $E$. Since we knew the length of the final polygon $P$ (this depended on the stopping time, but \emph{not} on the fineness of the zigzags) we also knew from \Cref{lemma:luzin} that during our movement, small enough subsets of $E$ will automatically cover small area. By the same reason, we could also ``ignore'' those $L_\ii$ for which $\cH^1(L_{\ii})\le\eps_\kk$, where $\kk$ is the last good sequence among $\ii$ and its ancestors. 

We now obtain the analogue estimates for rotations, by applying \Cref{lemma:luzin-rot} in place of  \Cref{lemma:luzin}. Indeed, since $|\tilde x_\jj| \lesssim  |x_\jj| = |v_\jj| = \cH^1(L_\jj)$ for each $\jj$, therefore every subset $R\subset E$ will cover, during the movement by $\tilde P$, an area $\lesssim c\cH^1(R)\sum_\jj|\tilde x_\jj|\lesssim c\cH^1(R)\sum_\jj\cH^1(L_\jj)= c\cH^1(R)\cH^1(P)$, where the sums are over all $\jj$ with $L_\jj \subset P$ (or, equivalently, $\tilde L_\jj \subset \tilde P$). That is, we obtain a $c$ times larger estimate than in  \Cref{lemma:luzin}. Similarly, when we move any $R$ by $\tilde L_\ii$, we cover an area at most $c\cH^1(R)|\tilde x_\ii|\lesssim c \cH^1(R)\cH^1(L_\ii)$ instead of $\cH^1(R)\cH^1(L_\ii)$.

Since $Q$ is a rotation, $|J_\ii| = |I_\ii|$. Similarly as in \Cref{sec:translations}, for each line segment $\tilde L \subset \tilde L_{\ii}$ appearing in the final polygon $\tilde P$, we choose $B(u_i, \eps)$ of \Cref{theorem:rotations-on-line} so that $\ell \setminus B(u_i, \eps)\subset J_\ii$ whenever $|J_\ii|=|I_\ii|\ge \pi-\eps$. If $|J_\ii|< \pi-\eps$, we can choose $B(u_i, \eps)$ arbitrarily.

\subsection{Area estimates}

For each $J$, let $A_J$ denote the set covered by moving $E_J$ along those $\tilde L_\ii \subset \tilde P$ for which $J\subset J_\ii$ (cf.\ \eqref{4.7}). Our final goal is to show $\sum_J |A_J| < \tilde c \eps$, for some $\tilde c$ independent of $\eps$. This would imply that \eqref{eq:theorem-rotations} holds with $\eps$ replaced by $\tilde c\eps$ in its right hand side. 

First we prove the following analogue of \Cref{lemma:lemm}. 

\begin{lemma}
\label{lemma:lemm-rot}
By making the basic zigzags sufficiently fine, we can achieve the following: if $J$ is an elementary interval contained in $[z_\ii, z_{\ii'}]$, then the area covered by moving $E_{J}$ along $\tilde L_\ii$ is $\lesssim c\alpha_\ii \cH^1(E_{J}) \cH^1(L_{\ii'})$.
\end{lemma}

\begin{proof}
Suppose $J$ is an elementary interval contained in $[z_\ii, z_{\ii'}]$. If $\ii$ is good, then there is a $\jj$ between $\ii$ and $\ii'$ such that $ J \subset [z_\jj, z_{\jj'}]$. If $\ii$ is bad, then for all $\jj$ between $\ii$ and $\ii'$, $J \subset [z_\jj, z_{\jj'}]$.

Suppose that  $J \subset [z_\jj, z_{\jj'}]$. Applying \Cref{var2} with $\rho=\rho(y_\jj)$ and $R = E_{J}$ (noting \Cref{remark:ball-contains-interval}), we see that if we move $E_{J}$ by the rotation $\rho(y_\jj)$, the area covered is 
$\lesssim c\alpha_\jj\cH^1(E_J)|\tilde x_\jj|/M_\jj$.
Hence the total area covered by moving $E_{J}$ by all $M_\jj$ copies of $\rho(y_\jj)$ is $\lesssim c\alpha_\jj\cH^1(E_J)|\tilde x_\jj| \lesssim c\alpha_\jj\cH^1(E_J) \cH^1(L_\jj)$.
We make the zigzags so fine in our constructions that the same estimate 
\begin{equation}
\label{eq:EQI-Ljj}\lesssim c\alpha_\jj\cH^1(E_J)\cH^1(L_\jj)
\end{equation}
remains true when we rotate the set $E_J$ by the descendants of the $M_\jj$ copies of $y_\jj$.

Now, we break into two cases. If $\ii$ is good, then $L_\ii$ descends from $L_\jj$, so the statement of the lemma follows from $\cH^1(L_\jj) \leq  \cH^1(L_{\ii'})$.

If $\ii$ is bad, we use the fact that $L_{\ii} = \bigcup_\jj L_\jj$ and $\tilde L_{\ii} = \bigcup_\jj \tilde L_\jj$, where the unions are over all $\jj$ between $\ii$ and $\ii'$. Then summing over the estimate \eqref{eq:EQI-Ljj} for each such $\jj$, we have that moving along $\tilde L_{\ii}$, the area is 
\[
\lesssim
c \alpha_\ii \cH^1(E_{J})
\sum_\jj
 \cH^1(L_\jj)
=
c \alpha_\ii \cH^1(E_{J})
\cH^1(L_\ii)
\leq
c \alpha_\ii \cH^1(E_{J})
\cH^1(L_{\ii'})
\]
which completes the proof.
\end{proof}

Having established this estimate, the proof continues in the same way as in \Cref{sec:translations}, to obtain $|A_{J}| \lesssim 2c \sum_{\ii} \eps_\ii \cH^1(E_{J})$, the analogue of \eqref{eq:AI-bound}. We explain some details below.

Consider an $\tilde L_\ii\subset \tilde P$ with $J\subset J_\ii$. Since the intervals $J_\ii,J_{\ii'},J_{\ii''},\dots$ are decreasing, there is a $\kk$ among $\ii$ and its ancestors such that $J\subset J_\kk\setminus J_{\kk'}\subset[z_\kk,z_{\kk'}]$. By \Cref{lemma:lemm-rot}, the total area covered when we move $E_J$ along $\tilde L_{\kk}$ is 
\begin{equation}\label{l''-rot}
\lesssim c\alpha_\kk \cH^1(E_J)\cH^1(L_{\kk'}).
\end{equation}

By making the zigzags sufficiently fine, the same estimate remains true when we move $E_J$ along all the descendants of $\tilde L_\kk$ in $\tilde P$.

Therefore, similarly as in section 4, the area of $A_J$ can be estimated by summing the estimate \eqref{l''-rot} for those ancestors that are on different family lines. Let $\kk_1,\kk_2,\dots$ be arbitrary sequences from different family lines.

We distinguish two cases: if $\kk_m$ is good, then
\begin{equation}\label{1-rot}
\alpha_{\kk_m}\cH^1(E_J)\cH^1(L_{{\kk_m}'})
\le\eps_{{\kk_m}'}\cH^1(E_J).
\end{equation}
With the bad $\kk_m$, because of the different family lines condition, each bad $\kk_m$ has a different ``last good among $\kk_m$ and its ancestors'' so
\begin{equation}\label{2-rot}
\sum_{\kk_m \text{ is bad}} \alpha_{\kk_m}\cH^1(E_J)\cH^1(L_{{\kk_m}'})
\le\sum_{\kk}\eps_{\kk}\cH^1(E_J),
\end{equation}
where the summation on the right is taken over all ${\kk}$. Adding together the estimates \eqref{1-rot} for all good $\kk_m$ and \eqref{2-rot}, we proved that
\begin{equation}
\label{eq:AJ-bound}
|A_J|\lesssim 2c\sum_{\kk}\eps_{\kk}\cH^1(E_J).\end{equation}

Let $c'$ be the implied constant in \eqref{eq:AJ-bound}. Since each $x\in E$ belongs to at most two of the sets $E_J$, we proved that $\sum_J|A_J|$ is at most $cc'$ times larger than the bound of $\eps$ for $\sum_I |A_I|$ that we obtained in \Cref{sec:translations}. In other words, we showed $\sum_J|A_J| < cc'\eps$. The constant $cc'$ depends only $\ell$ and $E$ (and not on $\eps$). This completes the proof.

\subsection{Further remarks}

\begin{remark} 
\label{remark:small-nbhd-initial-mvmt}
In both \Cref{sec:translations} and \Cref{sec:rotations}, we constructed a polygonal path that replaced a continuous movement with a fixed intrinsic projective center by a sequence of intrinsic rotations. By choosing all the zigzags sufficiently fine in our constructions, we can stay in an arbitrarily small neighborhood of the initial movement in $\Isom^+(\R^2)$. 
\end{remark}

\begin{remark} 
\label{remark:z-not-in-closure-ball}
It is possible to choose the $u_i$ in \Cref{theorem:rotations-on-line} so that $z$, the initial center of rotation, is not in any of the closed balls $\cl B(u_i, \eps)$. 

By applying \Cref{var2} to the initial rotation $\rho$ and a sufficiently small ball $B(z,\eta)$, we see that rotating the set $R = \{ x \in E : \nu_x \cap \ell \cap B(z, \eta)\}$ by $\rho$ covers small area. By making the zigzags sufficiently fine and using the small neighborhood lemma, the set $R$ still covers small area when moved by the final polygonal path. Thus, \eqref{eq:theorem-rotations} holds with $B(u_i,\eps)$ replaced by $B(u_i,\eps)\setminus B(z,\eta)$, so we can reselect the $u_i$ so that $z \not\in \cl B(u_i, \eps)$.

This property will be used in the proof of \Cref{theorem:besicovitch-rot}.
\end{remark}

\section{Besicovitch and Nikodym sets}
\label{section:besi-niko}

We conclude this paper by showing that when we iterate the polygonal constructions in \Cref{sec:translations} and \Cref{sec:rotations} and ``take the limit,'' we obtain the analogues of Besicovitch and Nikodym sets for rectifiable sets.

\subsection{Construction of a Besicovitch set for translations}
\label{subsec:besicovitch-translations}

We start with the following, somewhat technical conditions. Afterwards, we will discuss some interesting special cases.

Suppose that we are given some rectifiable sets $E_1\subset E_2\subset\dots$, and a tangent field $x\mapsto\theta_x$ of $\bigcup E_n$, satisfying the following:
\begin{enumerate}
\item each $E_n$ is compact, and has finite $\cH^1$-measure;
\item \label{item:besicovitch-E-n-prime} each $E_n$ has a subset $E_n'$ of full $\cH^1$-measure, such that the restriction of the tangent $\theta$ to $E_n'$ is continuous, and for each $y\in E_n$, 
\begin{equation}\label{thetay}
\theta_y\in\bigcap_{r>0}\cl(\theta(B(y,r)\cap E_n')).\end{equation}
\end{enumerate}

We will prove the following proposition:
\begin{proposition}\label{propo}
Suppose that the sets $E_n$ satisfy the assumptions above. Let $P_0$ be an arbitrary path in $\R^2$. Then for any neighborhood of $P_0$, there is a path $P$ in this neighborhood with the same endpoints as $P_0$, and there is a Borel mapping $p\mapsto\theta_p\in\P^1$ such that
\begin{equation}\label{pep}
|\bigcup_{p\in P} (p+\{x\in \bigcup E_n:\,\theta_x\neq \theta_p\}|=0.
\end{equation}
\end{proposition}

\begin{proof}
Given any neighborhood of $P_0$, let $P^0$ be a polygonal path in this neighborhood with the same endpoints as $P_0$. For each $n$, we choose an $\eps_n>0$ with $\sum\eps_n<\infty$. Then iteratively, for each $n\ge 1$ we apply \Cref{theorem:translations} to each segment $L \subset P^{n-1}$ with $E$ replaced by $E_n'$ and $\eps$ replaced by some $\eps_L > 0$ such that $\sum_{L \subset P^{n-1}} \eps_L < \eps_n$. This gives us a polygonal path $P^n=\bigcup_i L_i^n$ and directions $\theta_i^n$ such that
\begin{equation}\label{pn}
|\bigcup_i\bigcup_{p\in L_i^n}(p+E_i^n)|<\eps_n,
\end{equation}
where 
\begin{equation}\label{Ei}
E_i^n:= \cl{\{x\in E_n':\,\theta_x\not\in B(\theta_i^n,\eps_n)\}}.
\end{equation}

Although \Cref{theorem:translations} gives us the sets $E_i^n$ without their closure, we can take the closure in \eqref{Ei} since, by our assumptions, doing so does not increases their measure. (In particular, by assumption \eqref{item:besicovitch-E-n-prime}, we have $E_i^n \setminus \{x\in E_n':\,\theta_x\not\in B(\theta_i^n,\eps_n)\} \subset E_n \setminus E_n'$ and $\cH^1(E_n \setminus E_n') = 0$.) We know that moving an $\cH^1$-null set along a polygonal path covers only zero area, so indeed, \eqref{pn} holds. 

We construct $P^{n+1}$ by replacing each line segment $L_i^n$ of $P^n$ by a polygonal path that stays in such a small neighborhood of $L_i^n$ that the area estimate in \eqref{pn} remains true when, instead of $L_i^n$, we shift the sets $E_i^n$ along the line segments that we replace $L_i^n$ with. (Here we used \Cref{remark:small-nbhd-initial-mvmt} and that the sets $E_i^n$ are compact.)  

Also, we choose the neighborhoods small enough so that the polygonal paths $P^n$ converge to a continuous limit curve $P$. For each $p\in P$, and for each fixed $n$, we have an $i=i(p,n)$ such that $$|\bigcup_{p\in P}(p+E_{i(p,n)}^n)|<\eps_n$$ holds.
We denote 
\begin{equation}
\label{eq:def-Ep}
E^p:=\limsup_{n\to\infty}E_{i(p,n)}^n.
\end{equation} 
Then
\[
|\bigcup_{p\in P}(p+E^p)|
\le 
|\bigcup_{p\in P}(p+\bigcup_{m\ge n}E_{i(p,m)}^m)|
= 
|\bigcup_{m\ge n}\bigcup_{p\in P}(p+E_{i(p,m)}^m)|
\le
\sum_{m\ge n}\eps_m.
\]
Since this is true for every $n$, it follows that $\bigcup_{p\in P}(p+E^p)$ is Lebesgue null.

\medskip
By the definition \eqref{Ei}, if a point $y\in\bigcup E_n$ does not belong to $E^p$, then for every large enough $n$, it has a neighborhood disjoint from $\{x\in E_n':\,\theta_x\not\in B(\theta_i^n,\eps_n)\}$. That is, there is an $r>0$ such that $\theta_x\in B(\theta_i^n,\eps_n)$ for every $x\in B(y,r)\cap E_n'$. Hence, by our assumption \eqref{thetay}, $\theta_y\in\cl (\theta(B(y,r)\cap E_n'))\subset \cl B(\theta_i^n,\eps_n)$. That is, $\theta_y$ is in $\liminf_{n\to\infty}\cl B(\theta_{i(p,n)}^n,\eps_n)$, which has at most one point. For $p \in P$, if this set has one point, then we let $\theta_p$ denote that point. Otherwise, we let $\theta_p$ be arbitrary. 

Then for each $p$, $\{x\in \bigcup E_n:\,\theta_x \neq \theta_p\} \subset E^p$, and the proof is finished.
\end{proof}

For every rectifiable set $E$, we can choose the sets $E_n = E_n'$ such that they satisfy the requirements at the beginning of this section, and such that $\bigcup E_n$ is a subset of $E$ of full $\cH^1$-measure.
Therefore we obtain the following theorem:

\begin{theorem}[Besicovitch set for translations]
\label{theorem:besicovitch}
Let $E$ be an arbitrary rectifiable set, and let $x\mapsto\theta_x$ be an arbitrary tangent field of $E$. Then there is an $E_0\subset E$ of full $\cH^1$-measure in $E$ for which the following holds. 

For every path $P_0$ in $\R^2$, and for any neighborhood of $P_0$, there is a path $P$ in this neighborhood with the same endpoints as $P_0$, and there is a Borel mapping $p\mapsto\theta_p\in\P^1$ such that
\begin{equation}
|\bigcup_{p\in P} (p+\{x\in E_0:\,\theta_x\neq \theta_p\}|=0.
\end{equation}
\end{theorem}

\begin{remark}\label{rempre1}
Another interesting corollary of \Cref{propo} is the following. Suppose that $E$ can be covered by a finite union of (not necessarily disjoint) $C^1$ curves, or $E$ is the graph of a convex function. In these cases there is an $E_0\subset E$ of full measure so that the tangent is continuous on $E_0$. Moreover, we can define the tangent on $E\setminus E_0$ (in a natural way) and find the sets $E_n, E_n'$ so that they satisfy our requirements and so that $\bigcup_n E_n$ covers $E$. Therefore the statement of \Cref{theorem:besicovitch} holds with $E_0$ replaced by $E$.

For example, if $E$ is the graph of a strictly convex function, then it is enough to delete at most one point for each $p\in P$, as we claimed in the introduction.
\end{remark}

\subsection{Construction of a Besicovitch set for rotations}

The main ideas for rotations are the same as for translations. %

\begin{proposition}\label{propo-rot}
Suppose that the sets $E_n$ satisfy the assumptions as in the beginning of \Cref{subsec:besicovitch-translations}. Let $P_0$ be an arbitrary path in $\Isom^+(\R^2)$. Then for any neighborhood of $P_0$, there is a path $P$ in the neighborhood of $P_0$ with the same endpoints as $P_0$, and there is a Borel mapping $p \mapsto x_p \in \P^2$ such that 
\begin{equation}
|\bigcup_{p\in P} p(\{x\in \bigcup E_n :\, x_p \not\in \nu_x \})|=0.
\end{equation}
\end{proposition}

\begin{proof}
We begin with choosing $P^0$ to be an arbitrary polygonal path in the neighborhood of $P_0$ with the same endpoints as $P_0$. We iterate \Cref{theorem:rotations-on-line} to construct the polygonal paths $P^n$ in $\Isom^+(\R^2)$, each lying in a small neighborhood of the previous one. 
Here, the details are now a bit more technical, and we need to be careful when we specify our parameters for \Cref{theorem:rotations-on-line}.

As before, we choose an $\eps_n>0$ for each $n$ such that $\sum_n\eps_n<\infty$. Each line segment $L_i^n \subset P^n$ corresponds to a rotation $\rho_i^n$ with projective center $z_i^n$. We choose a line $\ell_i^n$ containing $z_i^n$ and a $0<\delta_i^n<\eps_n$. (We will impose additional conditions on $\ell_i^n,\delta_i^n$ in \Cref{hide}.) Then we replace $\rho_i^n$ by a sequence of intrinsic rotations by applying \Cref{theorem:rotations-on-line} and \Cref{remark:z-not-in-closure-ball} with $E$ replaced by $E_{n+1}'$, $\ell$ replaced by $\ell_i^n$, and $\eps$ replaced by $\delta_i^n$.

Choosing each of the parameters $\delta_i^n$ sufficiently small, we obtain the balls $B(u_i^n,\delta_i^n)$ and:
\begin{equation}
\label{6.8}
|\bigcup_i\bigcup_{p\in L_i^n}p(E_i^n)|<\eps_n,
\end{equation}
where 
\begin{equation}
E_i^n:=\cl\{x\in E_n'  : \nu_x\cap B(u_i^n,\delta_i^n)=\emp\}.
\end{equation}

We define $i(p, n)$ as in the previous section, and again take $E^p:=\limsup_{n\to\infty}E_{i(p,n)}^n$. Then as in the previous section, the movement $\bigcup_{p\in P} p(E^p)$ covers only a null set. 

Since $\delta_{i(p,n)}^n \to 0$, we know $\liminf_{n\to\infty}\cl B(u_{i(p,n)}^n,\delta_{i(p,n)}^n)$ can have at most one point.  If it has one point, let $x_p$ be that point. Otherwise, let $x_p$ be arbitrary.

Now suppose that $y\in E_n$ and $y \not\in E^n_{i(p,n)}$. Then $\nu_y \cap \cl B(u_{i(p,n)}^n,\delta_{i(p,n)}^n) \neq \emptyset$. Therefore indeed 
$\{x \in \bigcup E_n  :\, x_p \not\in \nu_x\} \subset E^p$, and the proof is finished.
\end{proof}

\subsection{The main theorem}
\label{hide}

In \Cref{propo-rot}, the points on $E$ that we hide at each $p \in P$ are those whose normal line passes through a particular point $x_p$. Since we would like to hide as little of $E$ as possible, it would be undesirable if an $x_p$ from our construction has the property that the normal line of positively many points of $E$ pass through $x_p$. 

Fortunately such points are very rare:

\begin{lemma}
There are at most countably many points with the property that the normal line of positively many points of $E$ pass through this point. 
\end{lemma}

\begin{proof}
Note that for any two such points there is only one common line, and there can be only an $\cH^1$-nullset of points of $E$ which have a given normal line. Since $E$ has $\sigma$-finite $\cH^1$-measure, it cannot have more than countably many subsets of positive measure such that their pairwise intersections are null.
\end{proof}

We denote the exceptional points above by $x_1,x_2,\dots$. In what follows, we show how to choose the parameters in our construction more carefully to avoid these points, i.e., so that $x_p \not\in \{x_1, x_2, \ldots \}$ for any $p \in P$.

We use the notation from the previous section. For each $n\ge 1$ and for each $L_i^n\subset P^n$, let $S_i^n$ denote the strip $B(\ell,\delta)$ assigned to \emph{the parent} of $L_i^n$, i.e., to the line segment in $P^{n-1}$ that we replaced by a polygon in the construction of $L_i^n$. Then we choose $\ell_i^n$, $\delta_i^n$ such that $B(\ell_i^n,\delta_i^n)\subset S_i^n$ and such that $\cl B(\ell_i^n,\delta_i^n)\setminus\{z_i^n\}$ does not contain any of the points $x_m$ with $m \leq n$.

Then  $\liminf_{n\to \infty} \cl B(u_{i(p, n)}^n, \delta_{i(p,n)}^n)$ is either empty, or contains one point. Suppose it contains a point $x_p$. Since $u_{i(p, n)}^n \in \ell_{i(p, n)}^n$ and the strips $\{B(\ell_{i(p, n)}^n,\delta_{i(p, n)}^n)\}_n$ are nested, it follows that $x_p \in \bigcap_n \cl B(\ell_{i(p, n)}^n,\delta_{i(p, n)}^n)$. By \Cref{remark:z-not-in-closure-ball}, $z_{i(p,n)}^n\not\in \cl B(u_{i(p, n+1)}^{n+1}, \delta_{i(p, n+1)}^{n+1})$, so $x_p \neq z_{i(p,n)}^n$ for all $n$.

Thus we have shown the following. This is the main theorem in our paper.

\begin{theorem}[Besicovitch set for rotations]
\label{theorem:besicovitch-rot}
Let $E$ be an arbitrary rectifiable set, and let $x\mapsto\theta_x$ be an arbitrary tangent field of $E$. Then there is an $E_0\subset E$ of full $\cH^1$-measure in $E$ for which the following holds. 

For every path $P_0$ in $\Isom^+(\R^2)$, and for any neighborhood of $P_0$, there is a path $P$ in the neighborhood of $P_0$ with the same endpoints as $P_0$, and there is a Borel mapping $p \mapsto x_p \in \P^2$ such that
\begin{equation}\label{pep-rot}
|\bigcup_{p\in P} p(\{x\in E_0:\, x_p \not \in \nu_x\})|=0.
\end{equation}
Furthermore, for each $p$, the set $\{x\in E_0:\, x_p \not \in \nu_x\}$ has full $\cH^1$-measure in $E$. 
\end{theorem}

\begin{remark}\label{rempre2} By the same argument as at the end of the previous section, we can get a stronger statement if the set $E$ has nice geometric properties. For instance, if it is covered by finitely many $C^1$ curves, or if it is the graph of a convex function, then the statement holds with $E_0$ replaced by $E$. 

As mentioned in the introduction, consider the special case where there is a line $\ell \in (\P^2)^*$ such that there is a neighborhood of $\ell$ in which no two normal lines of $E$ intersect. Then by choosing all the lines $\ell_i^n$ to lie inside this neighborhood, we can ensure that all the $x_p$ do as well. Hence, \Cref{theorem:besicovitch-rot} says that we can rotate $E$ continuously by $360^\circ$, covering a set of zero Lebesgue measure, where at each time moment, we only need to delete \emph{one point}.
\end{remark}

\begin{remark}
\label{remark:residual}
By the small neighborhood lemma, we can see that \eqref{6.8} holds (with the same sets $E_i^n$) not only for the path $P^n$ but for every continuous path $P$ sufficiently close to $P^n$. Using this observation, we obtain a dense open set of curves, and then, by taking the limit, a residual set of continuous paths $P$ connecting the endpoints of $P_0$, for which the statement of \Cref{theorem:besicovitch-rot} holds.
\end{remark}

\subsection{Construction of a Nikodym set}
\label{section:construction-nikodym}

We conclude this paper by explaining how the continuous Besicovitch sets can be used to construct Nikodym sets for rectifiable curves. 

Let $E \subset \R^2$ be an arbitrary rectifiable set. We fix an arbitrary (continuous) rectifiable curve $\Gamma \subset \R^2$ (if $E$ contains such a curve, we can choose $\Gamma$ to be that curve). By ``putting a copy of $E$ onto a point $y$,'' we mean that the corresponding copy of $\Gamma$ (i.e., the same isometry applied to $\Gamma$) goes through $y$.

\medskip
For every continuous rectifiable curve $\Gamma$, there is a path $P_0 \subset \Isom^+(\R^2)$ and a neighborhood of $P_0$ such that $\Gamma$ covers a set of non-empty interior along any path $P$ which lies in this neighborhood and has the same endpoints as $P_0$. (For example, if $\Gamma$ is a circle, we make sure that it is not possible for $P$ to be a rotation around the circle's center.) 

We apply \Cref{theorem:besicovitch-rot} with $E$ and with this neighborhood of $P_0$ to obtain a path $P$, and for each $p\in P$ to obtain a subset $E^p\subset E$ of full $\cH^1$-measure so that $|\bigcup_{p \in P} p(E^p)| = 0$.

	By our choice of $P_0$, we know that $\bigcup_{p \in P} p(\Gamma)$ has nonempty interior. Thus, $\bigcup_{q \in \Q^2} \bigcup_{p \in P} (q + p(\Gamma)) = \R^2$, whereas
\begin{equation}
\label{eq:nikodym}
A 
:=
\bigcup_{q \in \Q^2} \bigcup_{p \in P} (q + p(E^p))
\end{equation}
has measure zero.
Thus, we have shown the following.

\begin{theorem}
\label{theorem:nikodym}
Let $E$ be a rectifiable set and $\Gamma$ a rectifiable curve. Then the set $A$ defined by \eqref{eq:nikodym} is a \emph{Nikodym set} for $E$:
\begin{enumerate}
\item
$A$ has Lebesgue measure zero;
\item
Through each point $y \in \R^2$, $A$ contains a copy of $\cH^1$-a.e.\ point of $E$. That is, for all $y \in \R^2$, there is an $E_y \subset E$ and a $p_y \in \Isom^+(\R^2)$ such that $\cH^1(E \setminus E_y) = 0$, $y \in p_y(\Gamma)$, and $p_y(E_y) \subset A$.
\end{enumerate}
\end{theorem}

With \Cref{theorem:besicovitch} in place of \Cref{theorem:besicovitch-rot}
we can prove a result about placing \emph{translated} copies of $E$ at each point $y \in \R^2$. %

By essentially the same arguments as above, we now obtain a path $P \subset \R^2$ and $\theta_p \in \P^1$ such that $\bigcup_{p \in P} (p + E^p)$ has Lebesgue measure zero, where $E^p = \{x \in E_0 : \theta_x \neq \theta_p\}$, and such that $\bigcup_{p \in P} (p + \Gamma)$ has nonempty interior. Thus, $\bigcup_{q \in \Q^2} \bigcup_{p \in P} (q + p + \Gamma) = \R^2$, whereas
\begin{equation}
\label{eq:nikodym-translations}
A 
:=
\bigcup_{q \in \Q^2} \bigcup_{p \in P} (q + p + E^p)
\end{equation}
has Lebesgue measure zero. To ensure that $E^p$ has full $\cH^1$-measure in $E$, it is sufficient to assume that $\{x \in E : \theta_x = \theta\}$ is $\cH^1$-null for every $\theta \in \P^1$.

\begin{theorem}
\label{theorem:nikodym-translations}
Let $E$ be a rectifiable set and $\Gamma$ a rectifiable curve. Suppose that for every direction $\theta \in \P^1$, the set $\{x \in E : \theta_x = \theta\}$ is $\cH^1$-null. Then the set $A$ defined by \eqref{eq:nikodym-translations} satisfies the following:
\begin{enumerate}
\item
$A$ has Lebesgue measure zero;
\item
Through each point $y \in \R^2$, $A$ contains a translated copy of $\cH^1$-a.e.\ point of $E$. That is, for all $y \in \R^2$, there is an $E_y \subset E$ and a $p_y \in \R^2$ such that $\cH^1(E \setminus E_y) = 0$, $y \in p_y + \Gamma $, and $p_y + E_y \subset A$.
\end{enumerate}
\end{theorem}

\section{Dilations and similarity transformations}
\label{sec:dilations}

In this section, we show how the techniques of \Cref{sec:rotations} can be applied to analyze similarity transformations. Let $\Sim^+(\R^2)$ denote the space of all orientation-preserving similarity transformations in $\R^2$.

Elements in $\Isom^+(\R^2)$ were specified by the parameters $(w, \phi) \in \R^2 \times \R$. To index elements in $\Sim^+(\R^2)$, we introduce a new parameter $\alpha\in\R$. (In the special case of isometries we can take $\alpha=0$.)

\medskip
For $\alpha,\phi \in \R$, define $\phi_\alpha = e^{i\alpha} \phi$. For $\phi \neq 0$, we let $\rho_\alpha(w, \phi)$ denote the similarity transformation $u \mapsto e^{i \phi_\alpha} (u-z) + z$, where $z = w/\phi_\alpha$. Then it is natural to let $\rho_\alpha(w, 0)$ denote translation by $-iw$. For any $(w, \phi) \neq (0,0)$, we define the \emph{projective center} of $\rho_\alpha(w, \phi)$ to be the image of $(e^{-i\alpha} w, \phi) \in \R^3$ under the quotient map $\R^3\setminus \{0\} \to \P^2$.

\begin{remark}
The center of a translation now depends on $\alpha$. This is natural because a single translation can be viewed, e.g., as a rotation around some point at infinity and also as a dilation around some other point at infinity. 
\end{remark}

\begin{remark}\label{remark:remark2}
When $\alpha\equiv 0\ (\textrm{mod}\,\pi)$, the transformation $\rho_\alpha(w, \phi)$ is an isometry. When $\alpha \equiv \pi/2\ (\textrm{mod}\,\pi)$, the transformation is a dilation. For all other $\alpha$, 
the trajectory of a point $x$ under $\rho_\alpha(w, \phi)$ is a logarithmic spiral centered at $z$. Since
$$e^{i\psi_\alpha}=e^{i\psi(\cos\alpha+i\sin\alpha)}=e^{-\psi\sin\alpha}e^{i\psi\cos
\alpha},$$ the trajectory consists of those points $u$ for which $|u-z|=e^{-\psi\sin\alpha}|x-z|$ and $\arg (u-z)=\psi\cos\alpha+\arg(x-z)$ for some $\psi\in[0,\phi]$. For future reference, note that 
\begin{equation}\label{rem2}
\arg (u-z)=-\cot\alpha(\log |u-z|-\log|x-z|)+\arg(x-z).
\end{equation}
\end{remark}

\medskip

When studying similarity transformations, it turns out that instead of the normal line $\nu_x$, it is much more relevant to look at the normal line rotated by angle $\alpha$ around $x$. We denote this line by $(\nu_x)_\alpha$. We will prove the following generalization of \Cref{theorem:rotations-on-line}.

\begin{theorem}
\label{theorem:rotations-on-line-spiral}
Let $E \subset \R^2$ be a bounded rectifiable set of finite $\cH^1$-measure. Let $\epsilon > 0$, and let $\rho$ be a similarity transformation with parameter $\alpha$. Let $\ell\subset\P^2$ be a line through the projective center $z$ of $\rho$.

Then there are intrinsic similarity transformations $\rho_i=\rho_\alpha(x_i)$ with projective centers $z_i \in B(\ell, \epsilon) \subset \P^2$ such that the corresponding polygonal path $P=\bigcup_i L_i \subset \Sim^+(\R^2)$ connects the identity and $\rho$, and for each $i$, there exists a $u_i \in \ell$ such that 
\begin{equation}
\label{eq:theorem-rotations-spiral}
|\bigcup_{i}\bigcup_{p\in L_i} p(\{x\in E: (\nu_x)_\alpha \cap \ell \cap B(u_i, \eps)  = \emptyset\})|< \eps.
\end{equation}
Furthermore, if $\alpha \equiv \pi/2\ (\textrm{mod}\,\pi)$, then we can take $z_i \in \ell$.
\end{theorem}

Throughout this section, we fix an $\alpha \not\equiv 0 \pmod{\pi}$. To prove \Cref{theorem:rotations-on-line-spiral}, we first establish the analogues of \Cref{lemma:luzin} and \Cref{var2}. As in \Cref{sec:second-key-idea-rot}, assume $E \subset B(0, r) \subset \R^2$. We will show that there is a constant $c$ that depends only on $r$ such that the following two lemmas hold.

\begin{lemma}\label{lemma:luzin-spiral}
Let $y = (w, \phi) \in \R^2 \times \R$ with $|\phi| \lesssim 1$. Let $\rho=\rho_\alpha(y)$ be a similarity transformation and let $R\subset E$ be arbitrary. Then, if we transform $R$ by $\rho$, the area covered is $\lesssim c \cH^1(R)|y|$.
\end{lemma}

\begin{lemma}\label{var2-spiral}
Let $\delta>0$ be sufficiently small (depending on $r$).
Let $y = (w, \phi) \in \R^2 \times \R$ with $|\phi| \lesssim 1$. Let $\rho=\rho_\alpha(y)$ be a transformation with projective center $z$. Let $R\subset E$ be such that, for each $x\in R$, $(\nu_x)_\alpha \cap B(z,\delta)\neq \emptyset$. (Here, the ball $B(z, \delta)$ is defined with respect to the metric on $\P^2$.) Then, when we transform $R$ by $\rho$, the area covered is 
$$\lesssim c\delta\cH^1(R)|y|.$$
\end{lemma}

\begin{proof}[Proof of \Cref{lemma:luzin-spiral} and \Cref{var2-spiral}]
Let $\Psi:\,\R^2\to\R^2$ denote the measure preserving map that rotates each circle $|u-z|=\mathrm{const}$ around $z$ by angle $\cot\alpha\log |u-z|$. By \eqref{rem2}, $\Psi$ takes the spiral trajectories of $\rho$ to straight lines through $z$. In particular, $\Psi$ takes the trajectory of point $x$ under $\rho$ to the line segment 
\begin{equation}
\label{eq:line-segment}
[z+e^{i\cot\alpha\log|x-z|}(x-z),z+e^{-\phi\sin\alpha}e^{i\cot\alpha\log|x-z|}(x-z)].
\end{equation}
In other words, $z+\lambda e^{i\theta}$ belongs to this line segment if and only if $\theta=\cot\alpha\log|x-z|+ \arg(x-z)$ and $\lambda$ belongs to the interval $I_x \subset \R$, whose endpoints are $|x-z|$ and $e^{-\phi\sin\alpha}|x-z|$. (We do not specify which endpoint is the left and which is the right.)

Let $S \subset \R^2$ be the region covered by applying $\rho$ to $R$. We have
\begin{align*}%
|S| 
=
|\Psi(S)|
=
\int_0^{2\pi}
\int_{\{\lambda : z+\lambda e^{i\theta} \in \Psi(S)\}}
\lambda \, d\lambda \, d\theta
.
\end{align*}
To simplify the inner integral, observe that $\Psi(S)$ is the union of the line segments \eqref{eq:line-segment} over all $x \in R$. Thus
$$\int_{\{\lambda : z+\lambda e^{i\theta} \in \Psi(S)\}} \lambda \, d\lambda\leq
\sum_{x\in R \, :\, \arg (x-z) + \cot \alpha \log|x-z| = \theta}
\int_{I_x} \lambda \, d\lambda,$$
where 
$$\int_{I_x} \lambda \, d\lambda=|e^{-2\phi\sin\alpha} - 1|\cdot|x-z|^2\lesssim|\phi\sin\alpha|\cdot|x-z|^2.$$

Let $t \mapsto x(t)$ be a parametrization of $R$ by arclength. 
Note that the derivative of $t \mapsto \arg (x(t)-z) + \cot \alpha \log|x(t)-z|$ is
\[
\left< \frac{i(x(t) - z)}{|x(t)-z|^2}, \dot x(t) \right> + \cot\alpha \left<\frac{x(t) - z}{|x(t)-z|^2}, \dot x(t) \right>
.
\] 
Using the estimates above and the coarea formula, we have
\begin{align*}
|S|
&\lesssim
|\phi\sin\alpha|\int_0^{2\pi} \sum_{x\in R \, :\, \arg (x-z) + \cot \alpha \log|x-z| = \theta } |x-z|^2 \,d\theta
\\
&=
|\phi|
\int
\left<
e^{i\alpha}(x(t)-z)
,
\dot x(t)
\right>
\,dt
\\
&=
|\phi|
\int
\dist((\nu_x)_\alpha, z) \,d\cH^1(x)
.
\end{align*}

To prove \Cref{lemma:luzin-spiral}, we use the trivial estimate $\dist((\nu_x)_\alpha, z) \leq |x-z|$ and proceed as in the proof of \Cref{lemma:luzin-rot}. To prove \Cref{var2-spiral}, we proceed as in the proof of \Cref{var2}.
\end{proof}

For $x_1, x_2$, we define $x_3 = x_1 \star_{\alpha} x_2$ if $\rho_\alpha(x_3)$ can be replaced by $\rho_\alpha(x_1), \rho_\alpha(x_2)$. Explicitly, this means $\phi_1 + \phi_2 = \phi_3$ and $v_1 + e^{i(\phi_1)_\alpha}v_2 = v_3$, where $v_j = z_j(1-e^{i(\phi_j)_\alpha})$.

It is easy to check that with $\star_\alpha$ in place of $\star$, the arguments in \Cref{sec:rotations} still hold with very little modification, giving us a proof of \Cref{theorem:rotations-on-line-spiral}. (One small issue is that since the size of $E$ can change, we need to apply a correction factor to \Cref{lemma:luzin-spiral} and \Cref{var2-spiral}. However, we can ensure that at any point in the transformations, our set is never more than twice its initial size, so that the correction factor is bounded by an absolute constant.)

The statement at the end of \Cref{theorem:rotations-on-line-spiral} about $\alpha \equiv \pi/2 \pmod{\pi}$ follows from the fact that for such $\alpha$, if $x_3 = x_1 \star_{\alpha} x_2$, then the centers of the three dilations are collinear.

\subsection{Circles}\label{subsection:circles}

We briefly sketch the proof of \Cref{corollary:circles-dilation}.

\begin{proof}[Proof of \Cref{corollary:circles-dilation}.]
Let $E$ be a circle. Let $\epsilon> 0$ (to be specified later). By \Cref{theorem:translations}, there is a polygonal path $P = \bigcup_{i=1}^n L_i \subset \R^2$ with each $L_i$ a line segment, and for each $i$ there exists a direction $\theta_i\in\P^1$, such that
\begin{equation}
\label{eq:proof-circles-dilation-first-ineq}
|\bigcup_{i}\bigcup_{p\in L_i} (p+\{x\in E:\,\theta_x\not\in B(\theta_i,\eps)\})|< \epsilon.
\end{equation}

By \eqref{thetai}, we can assume that $\theta_{L_i} \not \in B(\theta_i, \eps)$ (recall that $\theta_{L_i}$ is the direction of the line segment $L_i$). By the fact that the tangent direction changes continuously as we move around the circle, there is an $\eps' < \eps$ such that 
\begin{equation}
|\bigcup_{i}\bigcup_{p\in L_i} (p+\{x\in E:\,\theta_x\not\in B(\theta_i,\eps')\})|< 2\eps.
\end{equation}
Since $\eps'<\eps$, we have $\theta_{L_i} \not \in \cl B(\theta_i, \eps')$.

Fix $\alpha = \pi/2$. Then $\theta_{L_i}$ is the projective center of the translation along $L_i$, and $(\nu_x)_\alpha$ is the tangent line at $x\in E$.

Let $\ell_i \in (\P^2)^*$ be the line through the center of $E$ of direction $\theta_{L_i}$. We fix an $i$, and for each $x\in E$ we denote by $\overline{x}$ the reflection of the point $x\in E$ across the line $\ell_i \cap \R^2$.

Let $\epsilon_i > 0$ (to be specified later). We can apply \Cref{theorem:rotations-on-line-spiral} to the translation along $L_i$, with $\alpha = \pi/2$ and line $\ell_i$, to replace this translation by a polygonal path of dilations $\bigcup_{j} L_{i,j} \subset \Sim^+(\R^2)$ such that 
\begin{equation}
\label{eq:bound-with-step2}
|\bigcup_{j}\bigcup_{p\in L_{i,j}} p(\{x\in E:\,(\nu_x)_\alpha\cap\ell_i\cap B(u_{i,j},\eps_i)=\emp\})|<\eps_i
\end{equation}
for some $u_{i,j} \in \ell_i$.

Those points $x$ of the circle $E$ for which the tangent line $(\nu_x)_\alpha$ intersects $\ell_i\cap B(u_{i,j},\eps_i)$ lie on circular arcs that are symmetric with respect to $\ell_i \cap \R^2$. Therefore we can find some $y_{i,j} \in E$ and $\eps_i'$ such that
\begin{equation}
\label{eq:bound-with-step2}
|\bigcup_{j}\bigcup_{p\in L_{i,j}} p(E\setminus(B(y_{i,j}, \eps_i')\cup B(\bar y_{i,j},\eps_i'))|< \eps_i.
\end{equation}

By the analogue of \Cref{remark:small-nbhd-initial-mvmt} and the small neighborhood lemma, we can also ensure that
\begin{equation}
\label{eq:bound-with-step1-small-nbhd}
|\bigcup_{i,j}\bigcup_{p\in L_{i,j}} p(\{x\in E:\,\theta_x\not\in B(\theta_i,\eps')\})|
< 3\eps.
\end{equation}

Since $\theta_{L_i} \not\in \cl B(\theta_i, \eps')$, it follows that if $r_i$ is small enough, then for each $y\in E$, $\{x\in E:\,\theta_x\in B(\theta_i,\eps')\} \cap (B(y,r_i)\cup B(\bar y,r_i))$ is either empty or is one arc of angle at most $\eps'$. Hence, if we choose $\eps_i$ so small that $\eps_i'<r_i$, then 
\[
E_{i,j}:=\{x\in E:\,\theta_x\not\in B(\theta_i,\eps')\}
\cup
(E\setminus(B(y_{i,j}, \eps_i')\cup B(\bar y_{i,j},\eps_i'))
\]
is a circular arc of angle at least $2\pi-\eps'$.

Thus, by combining \eqref{eq:bound-with-step2} and \eqref{eq:bound-with-step1-small-nbhd}, we have
\begin{equation}
|\bigcup_{i,j}\bigcup_{p\in L_{i,j}} p(E_{i,j})|
< 3\eps + \sum_i \eps_i.
\end{equation}

This gives us a movement of a subarc of $E$ of angle $2\pi-\eps'$ covering area less than $3\eps + \sum_i \eps_i$. Therefore by choosing the parameters small enough, we can move an arbitrarily large sub-arc covering arbitrarily small area.

Furthermore, as for isometries, we can construct not just one but a dense open set of movements. Therefore we can ensure that the radius of the circular arc remains very close to the radius of the original arc during the movement.
\end{proof}

By repeated applications of \Cref{theorem:rotations-on-line-spiral}, we obtain in the limit a Besicovitch set and a Nikodym set. The proofs proceed in the same way as in \Cref{section:besi-niko}. The result for Besicovitch sets in the special case when $E$ is a circle is stated below; for Nikodym sets, see \Cref{corollary:nikodym-circles-dilations}.

\begin{corollary}
Let $E$ be a circle. For every path $P_0$ in $\Isom^+(\R^2)$, and for any neighborhood of $P_0$ in $\Sim^+(\R^2)$, there is a path $P$ in the neighborhood of $P_0$ with the same endpoints as $P_0$, and there is a Borel mapping $p \mapsto x_p \in E$ such that
\[
|\bigcup_{p\in P} p(E\setminus\{x_p\})|=0.
\]
\end{corollary}

\providecommand{\bysame}{\leavevmode\hbox to3em{\hrulefill}\thinspace}
\providecommand{\MR}{\relax\ifhmode\unskip\space\fi MR }
\providecommand{\MRhref}[2]{%
  \href{http://www.ams.org/mathscinet-getitem?mr=#1}{#2}
}
\providecommand{\href}[2]{#2}


\begin{thebibliography}{CCHK18}

\bibitem[Bou86]{bourgain}
J.~Bourgain, \emph{Averages in the plane over convex curves and maximal
  operators}, J. Anal. Math. \textbf{47} (1986), 69--85. \MR{874045}

\bibitem[CCHK18]{cchk}
Alan Chang, Marianna Cs\"ornyei, Korn\'elia H\'era, and Tam\'as Keleti,
  \emph{Small unions of affine subspaces and skeletons via {B}aire category},
  Adv. Math. \textbf{328} (2018), 801--821. \MR{3771142}

\bibitem[CHL17]{chl}
M.~Cs\"ornyei, K.~H\'era, and M.~Laczkovich, \emph{Closed sets with the
  {K}akeya property}, Mathematika \textbf{63} (2017), no.~1, 184--195.
  \MR{3610009}

\bibitem[Cun74]{cunningham3}
F.~Cunningham, Jr., \emph{Three {K}akeya problems}, Amer. Math. Monthly
  \textbf{81} (1974), 582--592. \MR{0362055}

\bibitem[Dav71]{davies}
Roy~O. Davies, \emph{Some remarks on the {K}akeya problem}, Proc. Cambridge
  Philos. Soc. \textbf{69} (1971), 417--421. \MR{0272988}

\bibitem[Fal14]{falconer2014}
K.~J. Falconer, \emph{Fractal geometry}, third ed., John Wiley \& Sons, Ltd.,
  Chichester, 2014, Mathematical foundations and applications. \MR{3236784}

\bibitem[Fed69]{federer}
H.~Federer, \emph{Geometric measure theory}, Die Grundlehren der mathematischen
  Wissenschaften, Band 153, Springer-Verlag New York Inc., New York, 1969.
  \MR{0257325}

\bibitem[HL16]{hl}
K.~H\'era and M.~Laczkovich, \emph{The {K}akeya problem for circular arcs},
  Acta Math. Hungar. \textbf{150} (2016), no.~2, 479--511. \MR{3568105}

\bibitem[K{\"o}r03]{korner}
T.~W. K{\"o}rner, \emph{Besicovitch via {B}aire}, Studia Math. \textbf{158}
  (2003), no.~1, 65--78. \MR{2014552}

\bibitem[Mar87]{marstrand}
J.~M. Marstrand, \emph{Packing circles in the plane}, Proc. Lond. Math. Soc.
  (3) \textbf{55} (1987), no.~1, 37--58. \MR{887283}

\bibitem[Mat15]{mattila15}
P.~Mattila, \emph{Fourier analysis and {H}ausdorff dimension}, Cambridge
  Studies in Advanced Mathematics, Cambridge University Press, 2015.

\bibitem[Mit99]{mitsis}
T.~Mitsis, \emph{On a problem related to sphere and circle packing}, J. Lond.
  Math. Soc. (2) \textbf{60} (1999), no.~2, 501--516. \MR{1724841}

\bibitem[Ste76]{stein}
E.~M. Stein, \emph{Maximal functions. {I}. {S}pherical means}, Proc. Natl.
  Acad. Sci. USA \textbf{73} (1976), no.~7, 2174--2175. \MR{0420116}

\bibitem[Wol97]{wolff97}
T.~Wolff, \emph{A {K}akeya-type problem for circles}, Amer. J. Math.
  \textbf{119} (1997), no.~5, 985--1026. \MR{1473067}

\bibitem[Wol00]{wolff00}
\bysame, \emph{Local smoothing type estimates on {$L^p$} for large {$p$}},
  Geom. Funct. Anal. \textbf{10} (2000), no.~5, 1237--1288. \MR{1800068}

\end{thebibliography}
\end{document}